\documentclass[12 pt]{amsart}
\usepackage{amscd,amssymb,amsmath,amsthm}
\input xy
\xyoption{all}
\newtheorem{Proposition}{Proposition}
\newtheorem{Question}{Question}
\newtheorem{Lemma}{Lemma}
\newtheorem{Theorem}{Theorem}

\newcommand{\proj}{\mathbb{P}}

\newcommand{\rarr}{\rightarrow}
\newcommand{\oh}{{\mathcal{O}}}
\newcommand{\com}{\mathbb{C}}

\newcommand{\hpo}{{\hspace{1pt}}}

\newcommand{\bpf}{\noindent {\em Proof.} }
\newcommand{\epf}{\qed \vspace{+10pt}}

\newcommand{\Ext}{\text{Ext}}

\begin{document}
\baselineskip=16pt

\title{The moduli space of stable quotients}
\author{A. Marian, D. Oprea, and R. Pandharipande}
\date{March 2011}
\dedicatory{Dedicated to William Fulton on the occasion of his 70th
birthday}

\begin{abstract} A moduli space of stable quotients
of the rank $n$ trivial sheaf on stable curves is introduced.
Over nonsingular curves, the moduli space is
Grothendieck's Quot scheme. Over nodal curves, a 
relative construction is made to keep the torsion of
the quotient away from the singularities. 
New compactifications
of classical spaces arise naturally: a 
nonsingular and irreducible 
compactification of the moduli of maps from
genus 1 curves to projective space is obtained.
Localization on the moduli of stable quotients
leads to new relations in the
tautological ring 
generalizing Brill-Noether constructions.

The moduli space of stable quotients is proven to carry
a canonical 2-term obstruction theory and
thus a virtual class.
The resulting system of descendent invariants is
proven to equal the Gromov-Witten theory of
the Grassmannian in all genera. Stable quotients
can also be used to study Calabi-Yau geometries.
The conifold is calculated to agree with stable
maps. Several questions about the behavior of
stable quotients for arbitrary targets are raised.

\end{abstract}

\maketitle
\setcounter{tocdepth}{1}
\tableofcontents

\section{Introduction}
\subsection{Virtual classes}
Only a few compact moduli spaces in algebraic geometry
carry virtual classes. The conditions placed on the
associated deformation theories are rather strong.
The principal cases  (so far) are:
\begin{enumerate}
\item[(i)]
stable maps to nonsingular varieties \cite{BF,KKO,LT}, 
\item[(ii)]
stable sheaves
 on nonsingular 3-folds \cite{PT,T}, 
\item[(iii)]
stable sheaves
on nonsingular surfaces \cite {LT},  
\item[(iv)] Grothendieck's Quot scheme on 
nonsingular curves \cite{CFK,MO}.
\end{enumerate}
Of the above four families, the first three
are understood to be related. The correspondences
of \cite{MNOP1,MNOP2,PT} relate (i) and (ii). The
connections \cite{Park,Taub}
 between Gromov-Witten invariants and Donaldson/Seiberg-Witten
invariants  relate (i) and (iii). 
For equivalence with (ii) and (iii), the
associated  Gromov-Witten theories must be
considered with domains varying
in the moduli of stable curves $\overline{M}_{g}$.

The construction of the virtual class of the
Quot scheme (iv) requires the curve $C$ to be fixed
in moduli. In fact, the Quot scheme of a nodal curve
does {\em not} carry a virtual class via
the standard deformation theory. In order to fully
connect (i) and (iv), new moduli spaces are required.

\subsection{Stable quotients}
We introduce here a moduli space of {\em stable
quotients} 
$$\com^n \otimes \oh_C \rarr Q \rarr 0$$
on $m$-pointed curves $C$ with (at worst) nodal singularities. Two basic properties are
satisfied:
\begin{enumerate}
\item[$\bullet$] the quotient sheaf $Q$ is locally free at the nodes and markings of $C$,
\item[$\bullet$] the moduli of stable quotients
is proper over $\overline{M}_{g,m}$.
\end{enumerate}
The first property yields a virtual class, and the
second property leads to a system of 
invariants over $\overline{M}_{g,m}$.
Our main result equates the descendent theory of
the moduli of stable quotients to the Gromov-Witten
theory of the Grassmannian in all genera.

Stable quotients are defined
in Section \ref{defff}. The basic structures
of the moduli space
(including the virtual class) are discussed in
Section \ref{strrr}. The important case of
mapping to a point is studied in Section \ref{pttt}.
Comparison results with the Gromov-Witten theory
of Grassmannians in the strongest equivariant form are
stated in Section \ref{gwcomp}. The construction of
the moduli of stable quotients and
proofs of the comparison results are presented in
Section \ref{pppr} - \ref{ccccc}.

The intersection theory of the moduli of stable
quotients leads to new tautological relations
on the moduli of curves. Basic
relations generalizing classical
Brill-Noether constructions
are presented in Section \ref{ttaa6}.

Stable quotients can also be used to study
Calabi-Yau geometries. The most accessible are 
the local toric cases.
The conifold, given by the total space of
$$\mathcal{O}_{\proj^1}(-1) \oplus 
\mathcal{O}_{\proj^1}(-1) \rarr \proj^1,$$
is calculated in Section \ref{lllll} and found
to agree exactly with Gromov-Witten theory.

Given a projective embedding of an arbitrary scheme
 $$X\subset \proj^{n},$$
a moduli space of
stable quotients associated to $X$ is defined in Section \ref{varr3}.
We speculate,
at least when $X$ is a nonsingular complete intersection, that the 
moduli spaces carry virtual classes in all genera.
Virtual classes may exist in even greater
generality.

Stable quotient invariants in genus 1 for
Calabi-Yau hypersurfaces are discussed in Section \ref{elt}.
Let 
$$M_1(\proj^{n},d) \subset \overline{M}_1(\proj^{n},d)$$
be the open locus of the moduli of stable maps 
with nonsingular irreducible domain curves. 
Stable quotients
provide a nonsingular{\footnote{Nonsingularity
here is as a Deligne-Mumford stack.}}, 
irreducible, modular compactification
$$M_1(\proj^{n},d)\subset \overline{Q}_1(\proj^{n},d).$$
For the Calabi-Yau 
hypersurface of degree $n+1$,
$$X_{n+1}\subset \proj^n,$$
genus 1 invariants 
can be defined naturally as an Euler characteristic
of a rank $(n+1)d$ vector bundle on 
$\overline{Q}_1(\proj^{n},d)$. The relationship
to the Gromov-Witten invariants of $X_{n+1}$ is
not yet clear, but there will likely be a transformation.

The paper ends with
several questions about the behavior of 
stable quotients. 
Certainly, our main results carry over
to the hyperquot schemes associated to $\mathbf{SL}_n$-flag
varieties. Other variants are discussed in Section \ref{vart}. 
The toric case has been addressed in \cite{tor}.

\subsection{Later work}
Tautological relations coming from the stable quotient geometry, similar to those presented in Section  \ref{ttaa6}, are studied in \cite{PP}
on the moduli spaces $M_{g,n}^c$ of marked curves of compact type.  A Wick formalism is developed in order to evaluate the relations explicitly in terms of 
$\kappa$ classes.
The main results for $n>0$ are:
\begin{enumerate}
\item[(i)]
the $\kappa$ rings $\kappa^*(M_{g,n}^c)$ are generated by $\kappa$ classes of degree at most $g-1+ \lfloor \frac{n}{2} \rfloor $,
\item[(ii)]
there are no relations between the kappa classes below the threshold degree,
\item[(iii)] there is a natural isomorphism
$$ \kappa^*(M_{0,2g+n}^c)\stackrel{\sim}{\longrightarrow} \kappa^*(M_{g,n}^c)\ . $$
\end{enumerate}
Result (iii) is used to completely determine the $\kappa$ rings
including formulas for their Betti numbers.

A detailed study of the stable quotient relations on $M_g$ is undertaken 
in \cite{PPP}. The virtual class of the stable quotient space 
can be viewed as a new object in the classical theory of
linear systems on curves.
Using the Wick formalism and a series of transformations, 
  the stable quotient relations are recast to prove
an elegant set of relations in $R^*(M_g)$
conjectured by Faber and Zagier a decade ago.
Whether the Faber-Zagier relations are a complete set for
$R^*(M_g)$ is an interesting question. For $g\leq 23$, there
are no further relations. No further relations have been found
in any genus, but
by calculations of Faber, the set
does not yield a Gorenstein ring in genus 24. 

\vskip.05in

Finally, stable quotients should be considered to lie between
stable maps to the Grassmannian and 
stable sheaves relatively over $\overline{M}_g$
 \cite{PP3}.
Recent wall-crossing methods \cite{J,KS} will likely be relevant
to the study. A step in this direction is taken in \cite{Toda}: a series of moduli spaces is constructed, depending on a stability parameter 
and interpolating between the stable quotient and the stable map spaces.
Several further directions which have stable quotients as
their starting point are \cite{tor,C,Mano}.

\subsection{Acknowledgments}
The results presented here were obtained at MSRI in 
2009 
during a program on modern moduli in algebraic geometry. 
We thank the organizers
for creating a stimulating environment.
Conversations with D. Abramovich, I. Ciocan-Fontanine, Y. Cooper,
C. Faber, D. Maulik,
R. Thomas, and A. Zinger
have led to many improvements.

A.M. was partially supported by DMS-0812030. D.O. was partially
supported by DMS-0852468. R.P. was partially supported by DMS-0500187
and the Clay Institute. The paper was written while R.P. was
visiting the Instituto Superior T\'ecnico in Lisbon.

\section{Stability}\label{defff}
\subsection{Curves}
A {\em curve} is a reduced and connected scheme over
$\com$ of pure dimension 1.
Let $C$ be a curve of
arithmetic genus $$g= h^1(C,\oh_C)$$
with at worst nodal singularities.
Let $$C^{ns}\subset C$$
denote the nonsingular locus.
The data $(C,p_1,\ldots,p_m)$
with distinct markings $p_i\in C^{ns}$ determine a
genus $g$, $m$-pointed, {\em quasi-stable curve}.
A quasi-stable curve is {\em stable} if
$\omega_C(p_1+\ldots+p_m)$ is ample.

\subsection{Quotients}
Let $q$ be a quotient of the trivial bundle
 on a pointed quasi-stable curve $C$,
\begin{equation*}
\com^n \otimes \oh_C \stackrel{q}{\rarr} Q \rarr 0.
\end{equation*}
If $Q$ is locally free at the nodes and markings of $C$,
$q$ is a {\em quasi-stable quotient}.
Quasi-stability of $q$ implies
\begin{enumerate}
\item[(i)]
 the torsion subsheaf $\tau(Q)\subset Q$ 
has support contained in $$C^{ns}\setminus\{p_1,\ldots,p_m\},$$
\item[(ii)]
the associated
kernel,
\begin{equation*}
0 \rightarrow S \rightarrow
\com^n \otimes \oh_C \stackrel{q}{\rarr} Q \rarr 0,
\end{equation*}
is a locally free sheaf on $C$. 
\end{enumerate}
Let $r$ 
denote the rank of $S$.

Let $(C,p_1,\ldots,p_m)$ be a quasi-stable curve equipped
with a quasi-stable quotient $q$.
The data $(C,p_1,\ldots,p_m,q)$ determine 
a {\em stable quotient} if
the $\mathbb{Q}$-line bundle 
\begin{equation}\label{aam}
\omega_C(p_1+\ldots+p_m)
\otimes (\wedge^{r} S^*)^{\otimes \epsilon}
\end{equation}
is ample 
on $C$ for every strictly positive $\epsilon\in \mathbb{Q}$.
Quotient stability implies
$2g-2+m \geq 0$.

Viewed in concrete terms, no amount of positivity of
$S^*$ can stabilize a genus 0 component 
$$\proj^1\stackrel{\sim}{=}P \subset C$$
unless $P$ contains at least 2 nodes or markings.
If $P$ contains exactly 2 nodes or markings,
then $S^*$ {\em must} have positive degree.

Of course, when considering stable quotients in families,
flatness over the base is imposed on both the curve
$C$ and the quotient sheaf $Q$.

\subsection{Isomorphisms}
Let $(C,p_1,\ldots,p_m)$ be a quasi-stable curve.
Two quasi-stable quotients
\begin{equation}\label{fpp22}
\com^n \otimes \oh_C \stackrel{q}{\rarr} Q \rarr 0,\ \ \
\com^n \otimes \oh_C \stackrel{q'}{\rarr} Q' \rarr 0
\end{equation}
on $C$ 
are {\em strongly isomorphic} if
the associated kernels 
$$S,S'\subset \com^n \otimes \oh_C$$
are equal.

An {\em isomorphism} of quasi-stable quotients
 $$\phi:(C,p_1,\ldots,p_m,q)\rarr
(C',p'_1,\ldots,p'_m,q')
$$ is
an isomorphism of curves
$$\phi: C \stackrel{\sim}{\rarr} C'$$
satisfying
\begin{enumerate}
\item[(i)] $\phi(p_i)=p'_i$ for $1\leq i \leq m$,
\item[(ii)] the quotients $q$ and $\phi^*(q')$ 
are strongly isomorphic.
\end{enumerate}
Quasi-stable quotients \eqref{fpp22} on the same
curve $C$
may be isomorphic without being strongly isomorphic.

\begin{Theorem} The moduli space of stable quotients 
$\overline{Q}_{g,m}({\mathbb{G}}(r,n),d)$ parameterizing the
data
$$(C,\ p_1,\ldots, p_m,\  0\rarr S \rarr
\com^n\otimes \oh_C \stackrel{q}{\rarr} Q \rarr 0),$$
with {\em rank}$(S)=r$ and {\em deg}$(S)=-d$,
is a separated and proper Deligne-Mumford stack of finite type
over $\com$.
\end{Theorem}

Theorem 1 is obtained by mixing the construction of
the moduli of stable curves with the Quot scheme. 
Keeping the torsion of the quotient away from the nodes
and markings is a twist motivated by relative 
geometry.
The proof of Theorem 1 is given in Section \ref{pppr}.

\subsection{Automorphisms}

The automorphism group $\mathsf{A}_C$ of
a quasi-stable curve $(C,p_1,\ldots,p_m)$ 
may be positive dimensional.
If the dimension is $0$, $\mathsf{A}_C$
is finite.
Stability of $(C,p_1,\ldots,p_m)$ is well-known
to be equivalent to
the finiteness of $\mathsf{A}_C$.  
If $(C,p_1,\ldots,p_m,q)$ is a stable quotient,
the ampleness condition \eqref{aam}
implies that the marked curve $(C, p_1, \ldots, p_m)$ is semistable. Then, the connected component of the automorphism group $\mathsf{A}_C$
is a torus. \footnote {We assume $(g, m)\neq (1, 0)$.}

An {\em automorphism} of 
a quasi-stable quotient $(C,p_1,\ldots,p_m,q)$
is a  self-isomorphism. The automorphism
group  $\mathsf{A}_q$ of the quasi-stable quotient $q$
embeds in the automorphism group of the
underlying curve 
$$\mathsf{A}_q \subset \mathsf{A}_C.$$
We leave the proof of the following
elementary result to the reader.

\begin{Lemma}
Let $(C,p_1,\ldots,p_m,q)$ be a quasi-stable quotient such that $(C, p_1, \ldots, p_m)$ is semistable. Then $q$ is stable if and only if $\mathsf{A}_q$ is
finite.
\end{Lemma}

\subsection{First examples}
The simplest examples occur when $d=0$. Then, stability of
the quotient implies
the underlying pointed curve is stable. We see
$$\overline{Q}_{g,m}({\mathbb{G}}(r,n),0) = \overline{M}_{g,m} \times
{\mathbb{G}}(r,n)$$
where ${\mathbb{G}}(r,n)$ denotes the Grassmannian of $r$-planes
in $\com^n$.

A more interesting example is $\overline{Q}_{1,0}({\mathbb{G}}(1,n),1)$.
A direct analysis yields
$$\overline{Q}_{1,0}({\mathbb{G}}(1,n),1) = 
\overline{M}_{1,1} \times \proj^{n-1}.$$
Given a 1-pointed stable genus $1$ curve $(E,p)$ and
an element $\xi \in \proj^{n-1}$, the associated
stable quotient is 
$$0\rarr \oh_E(-p) \stackrel{\iota_\xi}{\rarr}
\com^n \otimes \oh_E \rarr Q \rarr 0$$
 where 
$\iota_\xi$ is the composition of the canonical inclusion
$$0\rarr \oh_E(-p) \rarr \oh_E$$
with the line in $\com^n$ determined by $\xi$.

The open locus ${Q}_{g,0}({\mathbb{G}}(r,n),d)\subset
\overline{Q}_{g,0}({\mathbb{G}}(r,n),d)$,
corresponding to nonsingular domains $C$,
is simply the universal Quot scheme over the moduli space
of nonsingular curves.

\section{Structures}\label{strrr}
\subsection{Maps}
Over the moduli space of stable quotients, there is a universal
curve
\begin{equation}\label{ggtt}
\pi: U \rarr \overline{Q}_{g,m}({\mathbb{G}}(r,n),d)
\end{equation}
with $m$ sections and a universal quotient
$$0 \rarr S_U \rarr \com^n \otimes \oh_U \stackrel{q_U}{\rarr} Q_U \rarr 0.$$
The subsheaf $S_U$ is locally free on $U$ because of the 
stability condition.

The moduli space $\overline{Q}_{g,m}({\mathbb{G}}(r,n),d)$ is equipped
with two basic types of maps.
If $2g-2+m >0$, then the stabilization of $(C,p_1,\ldots,p_m)$
determines a map
$$\nu:\overline{Q}_{g,m}({\mathbb{G}}(r,n),d) \rightarrow \overline{M}_{g,m}$$
by forgetting the quotient.
For each marking $p_i$, the quotient is locally free over $p_i$, hence it determines
an evaluation map
$$\text{ev}_i: 
\overline{Q}_{g,m}({\mathbb{G}}(r,n),d) \rightarrow {\mathbb{G}}(r,n).$$

The universal curve \eqref{ggtt} is {\em not} isomorphic to 
$\overline{Q}_{g,m+1}({\mathbb{G}}(r,n),d)$. In fact, there does {\em not}
exist a forgetful map of the form
$$\overline{Q}_{g,m+1}({\mathbb{G}}(r,n),d) \rarr \overline{Q}_{g,m}({\mathbb{G}}(r,n),d)$$
since there is no canonical way to contract the quotient sequence.

The general linear group $\mathbf{GL}_n(\com)$ acts on
$\overline{Q}_{g,m}({\mathbb{G}}(r,n),d)$ via 
the standard
action on $\com^n \otimes \oh_C$. The structures
$\pi$, $q_U$,
$\nu$ and the evaluations maps are all $\mathbf{GL}_n(\com)$-equivariant.

\subsection{Obstruction theory}
Even if $2g-2+m$ is not strictly positive, the moduli of stable
quotients maps to the Artin stack of pointed domain curves
$$\nu^A:
\overline{Q}_{g,m}({\mathbb{G}}(r,n),d) \rightarrow {\mathcal{M}}_{g,m}.$$
The moduli  of stable quotients with fixed underlying
curve 
$$(C,p_1,\ldots,p_m) \in {\mathcal{M}}_{g,m}$$
 is simply
an open set of the Quot scheme. 
The following result is obtained from the
standard deformation theory of the Quot scheme.

\begin{Theorem}\label{htr}
The deformation theory of the Quot scheme 
determines a 2-term obstruction theory on
$\overline{Q}_{g,m}({\mathbb{G}}(r,n),d)$ relative to
$\nu^A$
given by ${{RHom}}(S,Q)$.
\end{Theorem}

An absolute 2-term obstruction theory on
$\overline{Q}_{g,m}({\mathbb{G}}(r,n),d)$ is
obtained from Theorem \ref{htr} and the smoothness
of $\mathcal{M}_{g,m}$, see \cite{BF,GP}. The
 analogue of Theorem \ref{htr} for the Quot scheme of a {\it fixed} nonsingular
 curve was observed in \cite {CFK,MO}.

The $\mathbf{GL}_n(\com)$-action lifts to the
obstruction theory,
and the resulting virtual class is
defined in $\mathbf{GL}_n(\com)$-equivariant cycle theory,
$$[\overline{Q}_{g,m}({\mathbb{G}}(r,n),d)]^{vir} \in A_*^{\mathbf{GL}_n(\com)}
(\overline{Q}_{g,m}({\mathbb{G}}(r,n),d), \mathbb{Q}).$$
A system of 
${\mathbf{GL}_n(\com)}$-equivariant
descendent invariants is defined by the
brackets
$$\langle \tau_{a_1}(\gamma_1) \ldots \tau_{a_m}(\gamma_m) \rangle_{g,d}
= \int_{[\overline{Q}_{g,m}({\mathbb{G}}(r,n),d)]^{vir}}
\prod_{i=1}^m \psi_i^{a_i} \cup \text{ev}_i^*(\gamma_i)
$$
where $\gamma_i \in A^*_{\mathbf{GL}_n(\com)}      
({\mathbb{G}}(r,n),\mathbb{Q})$.
The classes $\psi_i$ are obtained from the cotangent
lines on the domain (or, equivalently, pulled-back from the
Artin stack by $\nu^A$).

\subsection{Nonsingularity}
Let $E$ be a nonsingular curve of genus 1, and let
$$f: E \rarr \mathbb{G}(1,n)$$
be a morphism of degree $d>0$.
The pull-back of the tautological sequence on 
$\mathbb{G}(1,n)$ determines a stable quotient on $E$.
The moduli space of maps
is an open{\footnote{If $d>1$, the subset is nonempty.
If $d=1$, the subset is empty.}} subset
\begin{equation}\label{xcfe}
M_{1,0}(\mathbb{G}(1,n),d) \subset 
\overline{Q}_{1,0}(\mathbb{G}(1,n),d)
\end{equation}
for $d>0$.

Let $(C,q)$ be a stable quotient parameterized by
$\overline{Q}_{1,0}(\mathbb{G}(1,n),d)$.
By stability, $C$ is either a nonsingular genus 1 curve or
a cycle of rational curves.
The associated sheaf $S$ is a line bundle of degree 
$-d<0$. The vanishing $$\text{Ext}^1(S,Q)=0$$
holds since 
there are no nonspecial
line bundles of positive degree on such curves.

\begin{Proposition} \label{nn245} 
$\overline{Q}_{1,0}({\mathbb{G}}(1,n),d)$
is a nonsingular irreducible Deligne-Mumford
stack  of dimension $nd$ for $d>0$.
\end{Proposition}

\begin{proof}
Nonsingularity has already been established.
The dimension is obtained from a Riemann-Roch
calculation{\footnote{The calculation is
done in general in Lemma \ref{kkop} below.}} of $\chi(S,Q)$.
Irreducibility is clear since 
${Q}_{1,0}({\mathbb{G}}(1,n),d)$ is
an open set of a projective bundle over the
moduli of elliptic curves.
\end{proof}

For simplicity, we will denote the moduli space
by 
$\overline{Q}_{1,0}({\mathbb{P}}^{n-1},d)$. 
Stable quotients
 provide an efficient 
compactification \eqref{xcfe} of $M_{1,0}(\proj^{n-1},d)$. Instead of desingularizing the moduli of 
maps by blowing-up  the closure 
of 
$$M_{1,0}(\mathbb{P}^{n-1},d)\subset \overline{M}_{1,0}(\proj^{n-1},d)$$
 in the
moduli of stable maps \cite{HL,VZ},
the stable quotient space achieves a simple modular
desingularization by {blowing-down}.

For large degree $d$, all line bundles on nonsingular
curves are nonspecial. As a result,
the following nonsingularity result holds.

\begin{Proposition}
For $g\geq 2$ and  $d\geq 2g-1$, the forgetful morphism 
$$\nu: Q_{g,0}(\mathbb P^{n-1}, d) \rarr M_g$$
is smooth of expected relative dimension.
\end{Proposition}

The result does {\em not} hold over the boundary or even over the 
interior if markings are present.

\section{Stable quotients for $\mathbb{G}(n,n)$}\label{pttt}
\subsection{$n=1$}
\label{xcx}
Consider $\overline{Q}_{g,m}(\mathbb{G}(1,1),d)$ for $d>0$. The moduli 
space parameterizes
stable quotients
$$0 \rarr S \rarr \oh_C \rarr Q \rarr 0.$$
Hence, $S$ is an ideal sheaf of $C$. 

Let $\overline{M}_{g,m|d}$ be the moduli space of genus $g$
curves with markings
$$ \{p_1\ldots,p_m\} \ \cup \
\{\widehat{p}_{1},\ldots,\widehat{p}_{d}\}
\in C^{ns}\subset C$$
satisfying the conditions
\begin{enumerate}
\item[(i)]
the points $p_i$ are distinct,
\item[(ii)]
the points $\widehat{p}_{j}$ are distinct from the points $p_i$,
\end{enumerate}
with stability given by the ampleness
of $$\omega_C(\sum_{i=1}^m {p}_{i} +\epsilon \sum_{j=1}^d  \widehat{p}_j)$$
for every strictly positive $\epsilon \in {\mathbb{Q}}$.
The conditions allow the points $\widehat{p}_{j}$ and $\widehat{p}_{j'}$ to coincide.

The moduli space  $\overline{M}_{g,m|d}$ is a nonsingular, irreducible,
Deligne-Mumford stack.{\footnote{In fact, $\overline{M}_{g,m|d}$
is a special case of the moduli of pointed curves with weights
studied by \cite{Has,LM}.}}
Given an element
$$[C,{p}_1,\ldots, {p}_m, \widehat{p}_1,
\ldots,\widehat{p}_d] \in \overline{M}_{g,m|d}\ , $$
there is a canonically associated stable quotient
\begin{equation}\label{jwq}
0 \rarr \oh_C(-\sum_{j=1}^d \widehat{p}_j) \rarr \oh_C \rarr Q \rarr 0.
\end{equation}
We obtain a morphism
$$\phi: \overline{M}_{g,m|d} \rarr \overline{Q}_{g,m}(\mathbb{G}(1,1),d).$$
The following result is proven by matching the
stability conditions.

\begin{Proposition}\label{xx23}
The map $\phi$ induces an isomorphism
$$\overline{M}_{g,m|d}/\mathbb{S}_d 
\stackrel{\sim}{\rarr} \overline{Q}_{g,m}(\mathbb{G}(1,1),d)$$
where the symmetric group $\mathbb{S}_d$ acts by permuting the markings
$\widehat{p}_j$.
\end{Proposition}

The first example to consider is $\overline{Q}_{0,2}(\mathbb{G}(1,1),d)$
for $d> 0$. The space has a rather simple geometry.
For example, the Poincar\'e polynomial
$$p_d = \sum_{k=0}^{2d-2} B_k t^k$$
where $B_k$ is the $k^{th}$ Betti number of
$\overline{Q}_{0,2}(\mathbb{G}(1,1),d)$, is
easily obtained.

\begin{Lemma} 
$p_d = (1+t^2)^{d-1}$ for $d>0$.
\end{Lemma}

\begin{proof} 
Let $(C,p_1,p_2,q)$ be an element of
$\overline{Q}_{0,2}(\mathbb{G}(1,1),d)$.
By the stability condition, $(C,p_1,p_2)$
must be a simple chain of rational curves with the
markings $p_1$ and $p_2$ on opposite extremal components.
We may stratify $\overline{Q}_{0,2}(\mathbb{G}(1,1),d)$
by the number $n$ of components of $C$ and the
distribution of the degree
 $d$ on these components. 
The associated
quasi-projective strata
$$S_{(d_1,\ldots,d_n)} \subset
\overline{Q}_{0,2}(\mathbb{G}(1,1),d)$$
are indexed by vectors
$$(d_1,\ldots, d_n), \ \ d_i >0, \ \ \sum_{i=1}^n d_i = d.$$
Moreover, each stratum is a product,
$$S_{(d_1,\ldots,d_n)} \stackrel{\sim}{=}
\prod_{i=1}^n (\text{Sym}^{d_i}(\com^*)/{\com^*}).$$

To calculate $p_d$, we must compute the virtual
Poincar\'e polynomial of the quotient space
$\text{Sym}^{k}(\com^*)/{\com^*}$
for all $k>0$.
We start with the virtual Poincar\'e polynomial of $\text{Sym}^k(\com)$, 
$$p(\text{Sym}^k(\com)) = p(\com^k) = t^{2k}.$$
Filtering by the order at $0\in \com$, we find 
$$p(\text{Sym}^k(\com)) = 
\sum_{i=0}^k p(\text{Sym}^i(\com^*)).$$
We conclude

$$p(\text{Sym}^k(\com^*)) = t^{2k}-t^{2k-2}$$
for $k>0$.
The quotient by $\com^*$ can be handled simply
by dividing by $t^2-1$, see \cite{GetPan}.
Hence,
$$p(\text{Sym}^{k}(\com^*)/{\com^*})= t^{2k-2}.$$
The Lemma then follows by elementary counting.
\end{proof}

\subsection{Classes}
There are several basic classes on 
$\overline{M}_{g,m|d}$.
As in the study of the standard moduli space of stable curves,
there are strata classes 
$$\mathcal{S}\in A^*(\overline{M}_{g,m|d},\mathbb{Q})$$
given by fixing the topological type of
a degeneration. New diagonal classes are defined for 
every subset $J\subset \{1,\ldots, d\}$ of size at least 2,
$$D_J \in A^{|J|-1} (\overline{M}_{g,m|d},\mathbb{Q}),$$
corresponding to the locus where the
points $\{\widehat{p}_j\}_{j\in J}$ are
coincident. In fact, the subvariety
$$D_J \subset \overline{M}_{g,m|d}$$
is isomorphic to $\overline{M}_{g,m|(d-|J|+1)}$.
The cotangent bundles
$$\mathbb{L}_i \rarr \overline{M}_{g,m|d}, \ \ \widehat{\mathbb{L}}_j  \rarr \overline{M}_{g,m|d}$$
corresponding to the two types of markings have respective
Chern classes
$$\psi_i=c_1(\mathbb{L}_i), \ 
\widehat{\psi}_j=c_1(\widehat{\mathbb{L}}_j) 
\in A^1(\overline{M}_{g,m|d}, \mathbb{Q}).$$
The Hodge bundle with fiber $H^0(C,\omega_C)$ over the curve $[C]\in \overline{M}_{g,m|d}$,
$$\mathbb{E} \rarr \overline{M}_{g,m|d},$$
has Chern classes
$$\lambda_i = c_i(\mathbb{E})\in A^i(\overline{M}_{g,m|d},\mathbb{Q}).$$

\subsection{Cotangent calculus}
Assume $2g-2+m\geq 0$. Canonical contraction defines a fundamental birational morphism
$$\tau: \overline{M}_{g,m+d} \rarr \overline{M}_{g,m|d}.$$
By the stability conditions, the cotangent lines at the points $p_i$ are
unchanged by $\tau$,
$$\tau^*(\psi_i) =\psi_i, \ \ \  1\leq i \leq m.$$
However, contraction affects the cotangent line classes
at the other points,
\begin{equation}\label{ppww}
\psi_{m+j}=\tau^*(\widehat{\psi}_j) + \Delta_{m+j}.
\end{equation}
Here, $\Delta_{m+j}$ is the sum 
$$\Delta_{m+j} = \sum_{j'\neq j} \Delta_{j, j'}$$
where $\Delta_{j, j'}$  is the boundary divisor of $\overline{M}_{g,m+d}$ 
parameterizing curves  
$$C= C'\cup C'', \ \ \ g(C')=0, \ \ g(C'') = g$$
with a single separating node  
and the markings labeled $m+j$ and $m+j'$ distributed to $C'$.

Let $\prod_{j=1}^d \widehat{\psi}_{j}^{\hpo y_j}$ be a monomial 
class on $\overline{M}_{g,m|d}$.
Since $\tau$ is birational,
\begin{equation}\label{drt7}
\tau_*\tau^*(\prod_{j=1}^d \widehat{\psi}_{j}^{\hpo y_j}) =\prod_{j=1}^d 
\widehat{\psi}_{j}^{\hpo y_j}.
\end{equation}
After using  relations \eqref{ppww} and \eqref{drt7}, we see for example
$$\tau_*(\psi_{m+j}) = 
\widehat{\psi}_j + \sum_{j'\neq j} D_{j,j'}\ .$$

\noindent The method proves the following result.

\begin{Lemma}\label{nbv}
There exists a universal formula 
$$\tau_*\left( \prod_{i=1}^m {\psi}_{i}^{x_i} \prod_{j=1}^d {\psi}_{m+j}^{y_j}\right)
= 
\prod_{i=1}^m {\psi}_{i}^{x_i}\left(  \prod_{j=1}^d 
\widehat{\psi}_{j}^{\hpo y_j}+ \ldots \right)$$
where the dots are polynomials in the $\widehat{\psi}_j$ and 
$D_J$ classes which are independent of $g$ and $m$.
\end{Lemma}

\subsection{Canonical forms} \label{can77}
Let $J,J' \subset \{ 1,\ldots, d\}$.
The cotangent line classes
\begin{equation}\label{k23}
\widehat{\psi}_j | D_J = \widehat{\psi}_J
\end{equation}
are all equal for $j\in D_J$.
If $J$ and $J'$ have nontrivial intersection, 
we obtain
\begin{equation}\label{p2h}
D_J \cdot D_{J'} =    (-\widehat{\psi}_{J\cup J'})^{|J\cap J'| -1} D_{J\cup J'}\ .
\end{equation}
by examining normal bundles.

If $M(\widehat{\psi}_j, D_J)$ is any monomial in the 
cotangent line and diagonal classes, we can write $M$ in a canonical
form in two steps:
\begin{enumerate}
\item[(i)] multiply the diagonal classes  using \eqref{p2h}
  until the result is a product of cotangent line classes
 with $D_{J_1} D_{J_2}\cdots D_{J_l}$ where all the subsets
$J_i$ are disjoint,
\item[(ii)] collect the equal cotangent line classes using \eqref{k23}.
\end{enumerate}
Let $M^{C}$ denote the resulting
canonical form.

By extending the  operation linearly, we can write any
polynomial $P(\widehat{\psi}_j, D_J)$ in canonical form $P^C$.
In particular, the universal formulas of Lemma 
\ref{nbv} can be taken to be in canonical form.

\subsection{Example}
The cotangent class intersections on $\overline{M}_{0,2|d}$,
\begin{equation}\label{heww}
\int_{\overline{M}_{0,2|d}} \psi_1^{x_1} \psi_2^{x_2} 
\widehat{\psi}_1^{\hpo y_1}
\cdots \widehat{\psi}_d^{\hpo y_d},
\end{equation}
for $d>0$ are straightforward to calculate.
Since the dimension of $\overline{M}_{0,2|d}$ is $d-1$,
at least one of the $y_j$ must vanish.
After permuting the indices, we may take $y_d=0$.
By studying the geometry of the map
$$\overline{M}_{0,2|d} \rarr \overline{M}_{0,2|d-1}$$
forgetting $\widehat{p}_d$ in case $d>1$,
we deduce
\begin{multline*}
\int_{\overline{M}_{0,2|d}} \psi_1^{x_1} \psi_2^{x_2} \widehat{\psi}_1^{\hpo y_1}
\cdots \widehat{\psi}_{d-1}^{\hpo y_{d-1}} = \\
\int_{\overline{M}_{0,2|d-1}} \psi_1^{x_1-1} \psi_2^{x_2} \widehat{\psi}_1^{\hpo y_1}
\cdots \widehat{\psi}_{d-1}^{\hpo y_{d-1}}
+\int_{\overline{M}_{0,2|d-1}} \psi_1^{x_1} \psi_2^{x_2-1} \widehat{\psi}_1^{\hpo y_1}
\cdots \widehat{\psi}_{d-1}^{\hpo y_{d-1}}.
\end{multline*}
Solving the recurrence, we conclude \eqref{heww}
vanishes unless all $y_j=0$ and
$$\int_{\overline{M}_{0,2|d}} \psi_1^{x_1} \psi_2^{x_2} =
\binom{d-1}{x_1,x_2}.$$

\subsection{Tautological complexes} \label{txtq}
Consider the universal curve
$$\pi: U \rarr \overline{M}_{g,m|d}$$
with universal quotient sequence
$$0 \rarr S_U \rarr \oh_U \rarr Q_U \rarr 0$$
obtained from \eqref{jwq}.
The complex 
$R\pi_*(S^*_U) \in D^b_{coh}(\overline{M}_{g,m|d})$
will arise naturally in localization calculations on
the moduli of stable quotients.
Base change of the complex to 
$$[C,{p}_1,\ldots, {p}_m, \widehat{p}_1,
\ldots,\widehat{p}_d] \in \overline{M}_{g,m|d}$$
computes the cohomology groups 
$$H^0(C, \oh_C(\sum_{j=1}^d \widehat{p}_j)), \ \
H^1(C, \oh_C(\sum_{j=1}^d \widehat{p}_j))$$
with varying ranks.

A canonical resolution by vector bundles of $R\pi_*(S^*_U)$
is easily obtained from the sequence
\begin{equation}\label{peww}
0 \rarr \oh_C \rarr \oh_C(\sum_{j=1}^d \widehat{p}_j) \rarr
\oh_C(\sum_{j=1}^d \widehat{p}_j)|_{\sum_{j=1}^d \widehat{p}_j} \rarr 0.
\end{equation}
The rank $d$ bundle
$$\mathbb{B}_d \rarr \overline{M}_{g,m|d}$$
with fiber
$$H^0(C,\oh_C(\sum_{j=1}^d \widehat{p}_j)|_{\sum_{j=1}^d \widehat{p}_j})$$
is obtained from the geometry of the points $\widehat{p}_j$.
The Chern classes of $\mathbb{B}_d$ are universal polynomials in the
$\widehat{\psi}_j$ and $D_J$ classes.
Up to a rank 1 trivial factor, $R\pi_*(S^*_U)$
is equivalent to the complex
$$\mathbb{B}_d \rarr \mathbb{E}^*$$
obtained from the derived push-forward of \eqref{peww}.

\subsection{General $n$}
While the moduli space
$$\overline{Q}_{g,m}(\mathbb{G}(1,1),d) \rarr \overline{M}_{g,m}$$
may be viewed simply as a compactification of the symmetric
product of the universal curve over $\overline{M}_{g,m}$,
the moduli space $\overline{Q}_{g,m}(\mathbb{G}(n,n),d)$ is more
difficult to describe since the stable subbundles have higher rank.
Nevertheless, since $\text{Ext}^1(S,Q)$ always vanishes, we obtain
the following result.

\begin{Proposition}
$\overline{Q}_{g,m}(\mathbb{G}(n,n),d)$ is nonsingular
of expected dimension $3g-3+m + nd$.
\end{Proposition}

\section{Gromov-Witten comparison}\label{gwcomp}
\subsection{Dimensions}
The moduli space of stable maps $\overline{M}_{g,m}({\mathbb{G}}(r,n),d)$
also carries a perfect obstruction theory and a virtual class.
In order to compare with the moduli space of stable quotients, we
will always assume  
$2g-2+m\geq 0$ and $0 < r < n$.

\begin{Lemma} \label{kkop}
The virtual dimensions of the spaces 
$\overline{M}_{g,m}({\mathbb{G}}(r,n),d)$ and 
$\overline{Q}_{g,m}({\mathbb{G}}(r,n),d)$ are equal.
\end{Lemma}

\bpf The virtual dimension of the moduli space of stable maps is
$$\int_{\beta} c_1(T) +  (\text{dim}_\com \ {\mathbb{G}}(r,n) -3) (1-g) +m 
=nd + ( r(n-r)-3)(1-g)+m.
$$
where $\beta$ is the degree $d$ curve class and $T$ is the
tangent bundle of ${\mathbb G}(r,n)$. 
Similarly, the virtual dimension of the moduli of stable
quotients is
$$\chi(S,Q) + 3g-3+m = nd+ r(n-r)(1-g)+3g-3+m,$$
by Riemann-Roch,
which agrees. 
\epf

\subsection{Stable maps to stable quotients}
There exists a natural morphism
$$c:\overline{M}_{g,m}({\mathbb{G}}(1,n),d)
\rightarrow 
\overline{Q}_{g,m}({\mathbb{G}}(1,n),d).$$
Given a stable map
$$f: (C,p_1,\ldots,p_m) \rarr \mathbb{G}(1,n)$$
of degree $d$, the image 
$c([f]) \in \overline{Q}_{g,m}({\mathbb{G}}(1,n),d)$
is obtained by the following construction.

The first step is to consider the minimal contraction
$$\kappa: C \rarr \widehat{C}$$
of rational components yielding
a quasi-stable curve $(\widehat{C},p_1,\ldots,p_m)$
with the automorphism group of each component
of dimension at most 1. The minimal contraction $\kappa$
is unique --- the exceptional curves of $\kappa$
are the maximal connected trees $T\subset C$ of rational curves which
\begin{enumerate}
\item[(i)] contain no markings, 
\item[(ii)] meet $\overline{C\setminus T}$ in a single point.
\end{enumerate}
Let $T_1, \ldots, T_t$ be the set of maximal trees
satisfying (i) and (ii). Then,
$$\widehat{C} = \overline{ C\setminus \cup_i T_i}$$
 is canonically a subcurve of $C$.
Let $x_1, \ldots, x_t \in \widehat{C}^{ns}$
be the points of incidence with the trees
$T_1, \ldots, T_t$
respectively.

Let $d_i$ be the degree of the restriction of $f$ to $T_i$.
Let
$$0 \rarr S \rarr \com^n\otimes \oh_{\widehat{C}} \rarr Q \rarr 0$$
be the pull-back by the restriction of
$f$ to $\widehat{C}$ of the tautological sequence
on $\mathbb{G}(1,n)$.
The canonical inclusion
$$0 \rarr S(-\sum_{i=1}^t d_i x_i) \rarr S$$
yields a new quotient
$$0 \rarr S(-\sum_{i=1}^t d_i x_i)
 \rarr \com^n\otimes \oh_{\widehat{C}} \stackrel{\widehat{q}}
\rarr \widehat{Q} \rarr 0.$$
Stability of the map $f$ implies $(\widehat{C},p_1,\ldots,p_m,\widehat{q})$
is a stable quotient. We define
$$c([f]) = (\widehat{C},p_1,\ldots,p_m,\widehat{q})\in
\overline{Q}_{g,m}({\mathbb{G}}(1,n),d).$$

The morphism $c$ has been studied earlier for genus $0$ curves
in the linear sigma model constructions of \cite{Giv}.
See Lemma $2.6$ of \cite{LLY} for a scheme theoretic discussion by J. Li.
The morphism $c$ is considered for the Quot scheme of
a fixed nonsingular curve of arbitrary genus in
\cite{Popa}.

\subsection{Equivalence}
The strongest possible comparison result holds
for $\mathbb{G}(1,n)$.

\begin{Theorem} 
$c_*[\overline{M}_{g,m}({\mathbb{G}}(1,n),d)]^{vir}=  
[\overline{Q}_{g,m}({\mathbb{G}}(1,n),d)]^{vir}$.
\end{Theorem}

If $r>1$, a morphism $c$ for $\mathbb{G}(r,n)$
does not in general exist.
However, the following construction provides a substitute.
Recall the Pl\"ucker embedding
$$\iota: {\mathbb{G}}(r,n) \rarr 
{\mathbb{G}}(1,\binom{n}{r}).$$
The Pl\"ucker embedding induces canonical maps
$$\iota_M:\overline{M}_{g,m}({\mathbb{G}}(r,n),d) \rarr
\overline{M}_{g,m}({\mathbb{G}}(1,\binom{n}{r}),d),$$
$$\iota_Q:\overline{Q}_{g,m}({\mathbb{G}}(r,n),d) \rarr
\overline{Q}_{g,m}({\mathbb{G}}(1,\binom{n}{r}),d).$$
The morphism $\iota_M$ is obtained by composing stable
maps with $\iota$. The morphism $\iota_Q$ is obtained by 
associating the subsheaf
$$0\rarr \wedge^r S \rarr \wedge^r \com^n \otimes \oh_C$$
to the subsheaf $0 \rarr S \rarr \com^n \otimes \oh_C$.

\begin{Theorem}\label{bestcom} 
For  $0<r<n$ and all classes 
$\gamma_i \in A^*_{\mathbf{GL}_n(\com)}
({\mathbb{G}}(r,n),\mathbb{Q})$,
\begin{multline*}
c_*\iota_{M*}\Big(\prod_{i=1}^m \text{\em ev}^*_i(\gamma_i) \ \cap \
[\overline{M}_{g,m}({\mathbb{G}}(r,n),d)]^{vir}\Big)= \\
  \iota_{Q*}\Big(\prod_{i=1}^m \text{\em ev}^*_i(\gamma_i) \ \cap \
[\overline{Q}_{g,m}({\mathbb{G}}(r,n),d)]^{vir}\Big).
\end{multline*}
\end{Theorem}

Since
descendent classes in both cases are easily seen to 
be pulled-back via $c\circ \iota_M$ and $\iota_Q$
respectively, there is no need to include them in
the statement of Theorem \ref{bestcom}.
In particular, Theorem \ref{bestcom} implies the
fully equivariant stable map and stable quotient
brackets (and CoFT) are equal.

\subsection{Example}
To see Theorems 3 and 4 are not purely formal, we can
study the case of genus 1 maps to $\proj^{n-1}$ of degree $1$ for $n\geq 2$.
Let 
$$I\subset \proj^{n-1} \times {\mathbb{G}}(2,n)$$
be the incidence correspondence consisting of points and
lines $(p,L)$ with $p\in L$.
The moduli space of stable maps is
$$\overline{M}_{1,0}(\proj^{n-1},1) = \overline{M}_{1,1} \times I.$$
We will denote elements of the moduli space of stable maps
by $(E, p, L)$
where $(E,p)\in \overline{M}_{1,1}$ and $(p,L)\in I$.
We have already seen
$$\overline{Q}_{1,0}(\proj^{n-1},1) = \overline{M}_{1,1} \times \proj^{n-1}.$$
The morphism
$$c:\overline{M}_{1,0}({\proj}^{n-1},1)
\rightarrow 
\overline{Q}_{1,0}(\proj^{n-1},1)$$
is given by the projection
$$I \rarr \proj^{n-1}$$
onto the first factor.
The virtual class of the moduli space of stable maps is
easily computed from deformation theory,
$$[\overline{M}_{1,0}(\proj^{n-1},1)]^{vir} = c_{n-2}(\text{Obs})\cap
[\overline{M}_{1,0}(\proj^{n-1},1)],$$
where the rank $n-2$ obstruction bundle is
$$\text{Obs}_{(E,p,L)}= \frac{{\mathbb{E}}^* \otimes T_p(\proj^{n-1})}
{\Psi_p^* \otimes T_p(L)}
= {\mathbb{E}}^* \otimes N_p(\proj^{n-1}/L)
\ . $$
Here, $\mathbb{E}$ is the Hodge bundle on $\overline{M}_{1,1}$,
$\Psi_p$ is the cotangent line, and 
$N_p(\proj^{n-1}/L)$ is the normal space to $L\subset \proj^{n-1}$
at $p$.
We see
\begin{multline*}
c_{n-2}(\text{Obs}) = c_{n-2}(N_p(\proj^{n-1}/L)) -\lambda
c_{n-3}(N_p(\proj^{n-1}/L))\\
 +
\lambda^2 c_{n-4}(N_p(\proj^{n-1}/L))+ \ldots\ 
\end{multline*}
where $\lambda=c_1(\mathbb{E})$.
Since $I \rarr \proj^{n-1}$ is a $\proj^{n-2}$-bundle,
\begin{eqnarray*}
c_*[\overline{M}_{1,0}(\proj^{n-1},1)]^{vir} & = &
c_*( c_{n-2}(N_p(\proj^{n-1}/L)) \cap 
[\overline{M}_{1,0}(\proj^{n-1},1)])
\\
& = &
[\overline{Q}_{1,0}(\proj^{n-1},1)] \\
& = &
[\overline{Q}_{1,0}(\proj^{n-1},1)]^{vir}.
\end{eqnarray*}
For the second equality,
we use the elementary projective geometry
calculation
$$c_{n-2}(Q)=1$$
where $Q$ is universal rank $n-2$ quotient on
the projective space of lines in $\com^{n-1}$.
The last equality follows since the moduli space of
stable quotients is nonsingular of expected dimension.

\section{Construction}
\label{pppr}

\subsection{Quotient presentation}
Let $g$, $m$, and $d$ satisfy
$$2g-2+m+\epsilon d >0$$
 for all $\epsilon >0$.
We will exhibit
the moduli space $\overline Q_{g,m}(\mathbb G(r,n), d)$ 
 as a quotient stack.  

To begin, fix a stable quotient $(C, p_1, \ldots, p_m, q)$ where 
$$0\to S\to \com^n\otimes \oh_C\stackrel{q}{\to} Q\to 0.$$ 
By assumption, the line bundle $$\mathcal L_{\epsilon}=
\omega_C(p_1+\ldots+p_m)\otimes (\Lambda^r S^{*})^{\epsilon}$$ 
is ample for all $\epsilon>0$. 
The
genus $0$ components of $C$ must
contain at least $2$ nodes or markings 
with strict inequality for components of degree $0$.
As a consequence, ampleness of $\mathcal L_{\epsilon}$ for 
$\epsilon=\frac{1}{d+1}$ 
is enough to ensure the stability of a degree $d$ quotient. 
We will fix $\epsilon=\frac{1}{d+1}$ throughout.  

By standard arguments, there exists a sufficiently large and 
divisible integer $f$
such that the line bundle 
$\mathcal L^f$ is very ample with no higher cohomology 
$$H^1(C, \mathcal L^f)=0.$$ 
We will show{\footnote{In fact, the result
is true for $k\geq 3$, but the arguments for $k\geq 5$
are simpler.}}
that for all $k\geq 5$, the choice
$$f=k(d+1)$$ 
suffices. Then,
$$\mathcal L^f=\left(\omega_C \left(\sum_{i=1}^m p_i\right)\right)^{k(d+1)}
\otimes\left (\Lambda^{r} S^{*}\right)^{k}.$$

To check very ampleness, 
we verify 
\begin{equation}\label{nn999}
H^1\left(C, \mathcal L^f \otimes I_{q_1} I_{q_2}\right)=0
\end{equation}
for all pairs of (not necessarily distinct) points
 $q_1, q_2\in C$. 
By duality, the vanishing \eqref{nn999}
is equivalent to 
$$\text {Ext}^0(I_{q_1} I_{q_2}, \omega_C\otimes \mathcal L^{-f})=0.$$ 
If $q_1, q_2\in C^{ns}$, 
we can check instead
 $$H^0(C, \omega_C(q_1+q_2)\otimes \mathcal L^{-f})=0,$$ 
which is clear since the line bundle has negative degree on each component. 
The following three cases also need to be taken into account:
\begin {enumerate}
\item [(i)] $q_1$ is node and $q_2\in C^{ns}$,
\item [(ii)] $q_1$ and $q_2$ are distinct nodes,
\item [(iii)] $q_1=q_2$ are  coincident nodes.
\end {enumerate} 
Cases (i-iii) can be easily handled. 
For instance, to check (i), consider the normalization at $q_1$, 
$$\pi:\widetilde C\to C,$$ 
and let $\pi^{-1}(q_1)=\{q_1', q_1''\}$. 
We have 
$$\text {Ext}^0(I_{q_1} I_{q_2}, \omega_C\otimes \mathcal L^{-f})=
H^0(\widetilde C, \omega_{\widetilde C}\otimes \pi^{*} 
\mathcal L^{-f}(q'_1+q''_1+q_2))$$ 
which vanishes since the line bundle on the right
has negative degree 
on each component. 
The other two cases are similar.
The condition $k\geq 5$ is used in (iii).

By the vanishing of the higher cohomology,
the dimension 
\begin{equation}\label{frt34}
h^0(C, \mathcal L^f)=1-g+k(d+1)(2g-2+m)+kd
\end{equation} 
is independent of the choice of stable quotient. 
Let $\mathsf V$ be a vector space of dimension \eqref{frt34}.
Given  
an identification 
$$H^0(C, \mathcal L^f)\cong \mathsf V^{*},$$ 
we obtain an embedding 
$$i:C\hookrightarrow \mathbb P(\mathsf V),$$ 
well-defined up to the action of the group ${\bf PGL}(\mathsf V)$. 
Let
$\mathsf{Hilb}$ denote the
Hilbert scheme of curves in $\mathbb P(\mathsf V)$ of 
genus $g$ and degree \eqref{frt34}. 
Each stable quotient gives rise to a point in 
$${\mathcal H}=\mathsf{Hilb}\times \mathbb P(\mathsf V)^{m},$$ where 
the last factors record the locations of the markings $p_1, \ldots, p_m$.

Elements of $\mathcal{H}$ are tuples
$(C,p_1,\ldots,p_m)$.
A quasi-projective subscheme ${\mathcal H}'\subset\mathcal{H}$ 
is defined by requiring 
\begin {enumerate}
\item [(i)] the points $p_1,\ldots,p_m$ are contained in $C$,
\item [(ii)] the curve $(C, p_1, \ldots, p_m)$ is quasi-stable.
\end {enumerate} 
We denote the universal curve over $\mathcal{H}'$ by
 $$\pi:\mathcal{C}'\to \mathcal{H}'.$$ 

 Next, we construct the $\pi$-relative Quot scheme 
$$\mathsf {Quot}(n-r,d)\rarr \mathcal{H}'$$ 
parametrizing rank $n-r$ degree $d$ quotients 
$$0\to S\to \com^n \otimes \oh_C \to Q\to 0$$ on the fibers of $\pi$. 
A locally closed subscheme $$\mathcal{Q}'
\subset \mathsf{Quot}(n-r,d)$$ is further singled out by requiring 
 \begin{enumerate}
\item[(iii)] $Q$ is locally free at the nodes and markings of $C$,
\item [(iv)] the restriction of $\mathcal O_{\mathbb P(\mathsf V)}(1)$ 
to $C$ agrees with the line bundle  
$$\left(\omega\left(\sum p_i\right)\right)^{k(d+1)}\otimes 
(\Lambda^{r}S^{*})^{k}.$$
\end {enumerate}

The action of ${\bf PGL}(\mathsf V)$ extends to $\mathcal{H}'$ 
and $\mathcal{Q}'$.
A  ${\bf PGL}(\mathsf V)$-orbit in $\mathcal{Q}'$ 
corresponds to a stable quotient up to isomorphism. 
By stability, each orbit has finite stabilizers. 
The moduli space $\overline Q_{g,m}(\mathbb G(r,n),d)$ 
is the stack quotient $\left[\mathcal{Q}'/\, {\bf PGL}(\mathsf V)\right]$.

\subsection{Separatedness} We prove
the moduli stack $\overline Q_{g,m}(\mathbb G(r,n), d)$ is separated 
by the valuative criterion. 

Let $(\Delta, 0)$ be a nonsingular pointed curve with complement 
$$\Delta^{0}=\Delta\setminus \{0\}.$$ 
We consider two flat families of 
quasi-stable pointed curves $$\mathcal X_i\to \Delta, \ \ \ \ \ \
p_1^i, \ldots, p_m^i:\Delta\to \mathcal X_i,$$ and two 
flat families of stable quotients 
$$0\to S_i\to \com^n \otimes
\mathcal{O}_{\mathcal{X}_i}\to Q_i\to 0,$$
for $1\leq i\leq 2$. 
We assume the two families are isomorphic
 away from the central fiber. We will show the isomorphism extends over $0$. 
In fact, by the separatedness of the Quot functor, 
we only need to show that 
the isomorphism extends to the 
families of curves $\mathcal X_i\to \Delta$ 
in a manner preserving the sections. 

By the semistable reduction theorem, 
possibly after base change ramified over $0$, 
there exists a third family 
$$\mathcal Y\to \Delta,\ \ \ \ \ \ p_1, \ldots, p_m:\Delta\to \mathcal Y$$ 
of quasi-stable pointed curves and dominant morphisms 
$$\pi_i:\mathcal Y\to \mathcal X_i$$ 
compatible with the sections. 
We may assume that $\pi_i$ restricts to an isomorphism away from the nodes of
$(\mathcal X_i)_0$. 

After pull-back, we obtain exact sequences of quotients
\begin{equation}\label{lw346}
0 \rarr \pi_i^*S_i \rarr 
\com^n \otimes
\mathcal{O}_{\mathcal{Y}}\to \pi_i^{*} Q_i\to 0
\end{equation}
on $\mathcal Y$ of the same degree and rank. 
Exactness holds after pull-back since the quotient $Q_i$ is locally free 
at the nodes of $(\mathcal X_i)_0$. 

The two pull-back sequences \eqref{lw346}
must agree on the central fiber by 
the separatedness of the Quot functor. 
We claim the central fiber $\mathcal Y_0$ cannot contain components 
which are contracted over the nodes of $(\mathcal X_1)_0$ but uncontracted 
over the nodes of $(\mathcal X_2)_0$. 
Indeed, if such a component $E$ existed, 
the quotient $\pi^{*}_1Q_1$ would be trivial on $E$, 
whereas by stability, 
the quotient 
$\pi^{*}_2Q_2$
could not be trivial. 
We conclude the
families $\mathcal X_1$ and $\mathcal X_2$ are isomorphic. \qed

\subsection {Properness} We prove the moduli stack 
$\overline Q_{g,m}(\mathbb G(r,n), d)$
is proper by the valuative criterion.
Let 
$$\pi^{0}:\mathcal X^{0} \to \Delta^{0},\ \ \ \ \ \  
p_1, \ldots, p_m:\Delta^{0}\to \mathcal X^{0}$$
carry a flat 
family of stable quotients 
\begin{equation}\label{ld56}
0\to S\to \com^n \otimes \mathcal{O}_{\mathcal{X}^0} \to Q\to 0
\end{equation}
which we must extend over $\Delta$, possibly after base-change. 
By standard reductions, after base change and normalization, 
we may assume the fibers of $\pi^0$ 
are nonsingular and irreducible curves, possibly after adding the 
preimages of the nodes to the marking set. 
The original family is reconstructed by gluing stable quotients 
on different components via the evaluation maps at the nodes. 

Once the general fiber of $\pi^0$
 is assumed to be nonsingular, 
we construct an extension $$\pi: \mathcal X\to \Delta,\ \ \ \ \ \
 p_1, \ldots, p_m:\Delta\to \mathcal X$$ 
with central fiber an $m$-pointed stable curve.{\footnote{There
are
exactly two cases where the central fiber can not be taken to stable,
$$(g,m)=(0,2) \ \ \text{or} \ \ (1,0)\ .$$
In both cases, the central fiber can be taken to be irreducible and
nodal. The argument afterwards is the same. We leave the details
to the reader.}}
After resolving the possible singularities of the total space  at the
nodes of $\mathcal{X}_0$ by blow-ups, we may take 
$\mathcal X$ to be a nonsingular surface. 
Using the properness of the relative Quot functor, 
we complete the family of quotients across the central fiber:
$$0\to S\to \com^n \otimes \mathcal{O}_{\mathcal{X}}\to Q\to 0.$$
The extension may fail to be a quasi-stable quotient since
$Q$ may not be locally free at the nodes or the markings of the central fiber. This will be corrected by further blowups. 

We will first
treat the case when $S$ has rank $1$. 
As explained in Lemma $1.1.10$ of \cite {OMS}, 
the sheaf $S^{*}$ is reflexive over the nonsingular
 surface $\mathcal X$, hence locally free. Consider the image $T$ of the map 
$$(\com^n\otimes \mathcal{O}_{\mathcal{X}})^{*}\to S^{*}$$ which 
can be written as $$T=S^{*}\otimes I_Z$$ for a subscheme 
$Z\subset\mathcal X$. 
The quotient $Q$ will have torsion supported on $Z$.  
By the flatness of $Q$, the subscheme $Z$ is not supported
on any components of the central fiber.

We consider 
a point $\xi\in\mathcal{X}$ which is a node or marking of the 
central
fiber.
After restriction to an open set containing $\xi$, 
we may assume all components of $Z$ pass through $\xi$. 
After a
 sequence of blow-ups 
$$\mu:\widetilde {\mathcal X}\to \mathcal X,$$ 
we may take 
$$\widetilde Z=\mu^{-1}(Z) =\sum_i m_i E_i + \sum_j n_j D_j,$$
where the $E_i\subset \mathcal{X}$ are
the exceptional curves of $\mu$ and
the $D_j$ intersect the $E_i$  
away from the nodes and markings. 
Since we are only interested in constraining the behavior of
$\widetilde{Z}$ at the nodes or markings over $\xi$, 
the morphism $\mu$
can be achieved by repeatedly blowing-up only {nodes} or
{markings} of the fiber over $\xi$.

On the blow-up, the image of the map 
$$(\com^n\otimes \mathcal{O}_{\mathcal{X}})^{*}
 \to \mu^{*}S^{*}$$ factors 
though $\mu^{*} S^{*}(-\widetilde Z).$ 
Setting $$\widetilde S=\mu^{*} S(\sum_i m_i E_i)\hookrightarrow 
\com^n \otimes \mathcal{O}_{\widetilde{\mathcal{X}}},$$
 we obtain a flat family 
\begin{equation}\label{wjj45}
0\to \widetilde S\to \com^n \otimes \mathcal{O}_
{\widetilde{\mathcal{X}}}\to \widetilde Q\to 0
\end{equation}
on $\widetilde {\mathcal X}$  
where the quotient $\widetilde Q$ is locally free 
at the nodes or the markings of the (reduced) central fiber.

Unfortunately, the above blow-up process yields a family
$$\widetilde{\mathcal{X}}\rarr \Delta$$
with possible nonreduced components 
occuring in chains over nodes and markings of $\mathcal{X}_0$.
The multiple components can be removed by 
base change and normalization,
$${\mathcal{X}}' \rarr \widetilde{\mathcal{X}},$$
with the nodes and markings of $\mathcal{X}'_0$ mapping to
the nodes and markings of $\widetilde{\mathcal{X}}_0^{red}$.

The pull-back of 
\eqref{wjj45} to ${\mathcal{X}}'$
yields a quotient
$$\com^n \otimes \mathcal{O}_
{{\mathcal{X}}'}\to Q'\to 0.$$
The quotient $Q'$ is certainly locally free (and hence flat)
over the nodes and markings 
of ${\mathcal{X}}'_0$. 
The quotient $Q'$ may fail to be flat over finitely
many nonsingular points of ${\mathcal{X}}'_0$.
A flat limit
\begin{equation}\label{jj33}
0 \rarr S'' \rarr \com^n \otimes \mathcal{O}_
{{\mathcal{X}}'}\to  Q''\to 0
\end{equation}
can then be found after altering $Q'$ only at the latter
points. 
Since $Q''$ is locally free over the nodes and markings of
$\mathcal{X}'_0$, we have constructed a quasi-stable quotient.
However \eqref{jj33}
may fail to be stable because of possible unstable genus $0$ components 
in the central fiber. 

By the economical choice of blow-ups (occuring only at nodes
and markings over $\xi$),
all unstable genus 0 curves $P$ carry exactly 2 special points 
and
$${S}''|_P \stackrel{\sim}{=} \oh_P.$$
All such unstable components
are contracted by the line bundle 
$$\mathcal L=\omega_C(p_1+\ldots+p_m)^{d+1}\otimes \Lambda^r 
(S'')^{*}.$$ 
Indeed, $\mathcal L^k$ is $\pi'$-relatively\footnote{Here, ${\pi}': 
{\mathcal X}' \rarr \Delta$.}
basepoint free
for 
$k\geq 2$ and trivial over the unstable genus 0 curves.
As a consequence, $\mathcal L^k$ determines a morphism 
$$q:{\mathcal X}'\to \mathcal Y=\text {Proj}\left(\oplus_{m} L^{km}\right).$$ The push-forward 
$$0\to q_{*} S'' \to \com^n \otimes \mathcal{O}_{\mathcal{Y}}
 \to q_{*}  Q''\to 0$$ is stable.
 We have constructed the limit of the original
family \eqref{ld56} of stable quotients over $\Delta^{0}$.

The case when the subsheaf $S$ has arbitrary rank is similar. 
The cokernel $K$ of the map 
$$(\com^n \otimes \mathcal{O}_{\mathcal{X}})^{*}\to S^{*}$$ 
has support
 of dimension at most $1$. 
The initial Fitting ideal of $K$, denoted $\mathcal F_0(K)$, 
endows the support of $K$ with
a natural scheme structure. 
After a suitable composition of blow-ups 
 $$\mu:\widetilde {\mathcal X}\to \mathcal X,$$
 we may  take
$$\mathcal F_0(p^*K)=p^{*} \mathcal F_0(K)$$ 
to be divisorial with only exceptional components
passing through the nodes and markings of the central fiber.
Let $V$ be the exceptional part of $\mathcal F_0(p^*K)$
We set $$K'=\mu^{*}K\otimes \mathcal O_{V},$$ 
and define the sheaves $\widetilde K$ and $\widetilde S$ 
by the diagram
$$\xymatrix {(\com^n \otimes \mathcal{O}_{\widetilde{\mathcal{X}}})^{*} 
\ar[r]\ar@{=}[d] & \widetilde{S}^*\ar[r]\ar[d] & \widetilde{K}\ar[d] 
\\ (\com^n\otimes \mathcal{O}_{\mathcal{X}})^{*} 
\ar[r] & \mu^{*} S^{*}\ar[r]\ar[d]& \mu^{*}K\ar[d]\\ & K' \ar@{=}[r]& K'}.$$ 
The Fitting ideal 
$\mathcal F_0(\widetilde K)$ does
not vanish on exceptional divisors of $\mu$. 
Therefore, 
the quotient $$0\to \widetilde S \to 
\com^n\otimes \mathcal {O}_{\mathcal{\widetilde{X}}} 
\to \widetilde Q\to 0$$ is locally free at  the nodes or the markings
of the (reduced) central fiber. The remaining steps
exactly follow the 
 rank $1$ case.\qed

\section{Proofs of Theorems 3 and 4} \label{ccccc}
\subsection{Localization}
The idea 
is to proceed by localization 
 with respect to
the maximal torus $\mathbf{T}\subset {\mathbf{GL}_n(\com)}$
acting with diagonal weights $\mathsf{w}_1,\ldots,\mathsf{w}_n$. 
By the usual splitting principle, the torus calculation is enough
for the full equivariant result. Localization formulas for the virtual classes 
 of the moduli of  stable maps and
stable 
quotients{\footnote{An 
analogous localization computation for the virtual class of the Quot scheme of a fixed nonsingular
 curve was carried out in \cite {MO}. In particular, the fixed loci and their contributions were explicitly determined. The localization for the stable quotient  space is conceptually similar.}}
are both given by \cite{GP}.

Theorem 3 is a special case of Theorem 4, so will
consider only the latter. We will 
compare fixed point residues pushed-forward to 
$$\overline{Q}_{g,m}({\mathbb{G}}(1,\binom{n}{r}),d).$$

\subsection{$\mathbf{T}$-fixed loci for stable maps}
\label{fff1}
The $\mathbf{T}$-fixed loci of the moduli space
$\overline{M}_{g,m}(\mathbb{G}(r,n),d)$ are described in detail in \cite {GP}. We briefly recall here that the fixed loci are indexed by decorated graphs
$(\Gamma, \nu,\gamma,\epsilon, \delta, \mu)$ where
\begin{enumerate}
\item[(i)] $\Gamma=(V,E)$ such that $V$ is the vertex set and $E$ is the edge set (with no
self-edges),
\item[(ii)] $\nu:V \rarr \mathbb{G}(r,n)^{\mathbf{T}}$ is an assignment
of a $\mathbf{T}$-fixed point $\nu(v)$ to each element $v\in V$,
\item[(iii)] $\gamma: V \rarr \mathbb{Z}_{\geq 0}$ is a genus 
 assignment,
\item[(iv)] $\epsilon$ is an assignment to each $e\in E$ of a
 $\mathbf{T}$-invariant curve $\epsilon(e)$ of $\mathbb{G}(r,n)$ together with a
covering number $\delta(e) \geq 1$,
\item[(v)] $\mu$ is a distribution of the $m$ markings to the vertices $V$.
\end{enumerate}
The graph $\Gamma$ is required to be connected.
The two vertices incident to the edge $e\in E$ must
correspond via $\nu$ to the two $\mathbf{T}$-fixed points
incident to $\epsilon(e)$. The sum of $\gamma$ over $V$ together with
$h^1(\Gamma)$ must equal $g$.
The sum of $\delta$ over $E$ must equal $d$.

The $\mathbf T$-fixed locus corresponding to a given graph is, up to automorphisms, the product $$\prod_{v} \overline M_{\gamma(v), \,\text{val}(v)},$$ where $\text{val}(v)$ counts all incident edges and markings. The stable maps in the $\mathbf T$-fixed locus are easily described. If the condition 
$$2\gamma(v)-2 + \text{val}(v) >0$$
 holds{\footnote{Otherwise, the vertex is {\em degenerate}.}}, 
then the vertex $v$ corresponds to a collapsed curve
varying in $\overline{M}_{\gamma(v), \text{val}(v)}$. 
Moreover, each edge $e$ gives a degree $\delta(e)$ covering of the invariant curve $\epsilon(e),$ ramified only over the two torus fixed points. The stable map is obtained by gluing along the graph incidences.

\subsection{$\mathbf{T}$-fixed loci for stable quotients}
\label{fff2}

\subsubsection{The indexing set} The $\mathbf{T}$-fixed loci of 
 $\overline{Q}_{g,m}(\mathbb{G}(r,n),d)$
are similarly indexed by decorated graphs
$(\Gamma, \nu,\gamma, s,\epsilon, \delta, \mu)$ where 
\begin{enumerate}
\item[(i)] $\Gamma=(V,E)$ such that $V$ is the vertex set and $E$ is the edge set (no
self-edges are allowed),
\item[(ii)] $\nu:V \rarr \mathbb{G}(r,n)^{\mathbf{T}}$ is an assignment
of a $\mathbf{T}$-fixed point $\nu(v)$ to each element $v\in V$,
\item[(iii)] $\gamma: V \rarr \mathbb{Z}_{\geq 0}$ is a genus 
 assignment,
\item[(iv)] $s(v)=(s_1(v),\ldots,s_r(v))$ is an assignment of a
tuple of non-negative integers with $\mathbf{s}(v)= \sum_{i=1}^r s_i(v)$
together with an inclusion  $$\iota_s: \{1,\ldots,r\} \rarr \{1,\ldots,n\},$$
\item[(v)] $\epsilon$ is an assignment to each $e\in E$ of a
 $\mathbf{T}$-invariant curve $\epsilon(e)$ of $\mathbb{G}(r,n)$ together with 
a covering number $\delta(e) \geq 1$,
\item[(vi)] $\mu$ is a distribution of the markings to the vertices $V$.
\end{enumerate}

The graph $\Gamma$ is required to be connected.
The two vertices incident to the edge $e\in E$ must
correspond via $\nu$ to the two $\mathbf{T}$-fixed points
incident to $\epsilon(e)$. The sum of $\gamma$ over $V$ together with
$h^1(\Gamma)$ must equal $g$.
The assignment $s$ determines the splitting type of the
subsheaf over the vertex $v$. The inclusion $\iota_s$
determines $r$ trivial factors of $\com^n \otimes \oh_C$ in which the subsheaf $S$ injects. 
The inclusion $\iota_s$ must be compatible with $\nu(v)$.
The sum of $\mathbf{s}(v)$ over $V$ together with 
the sum of $\delta$ over $E$ must equal $d$.

Unless $v$ is a degenerate vertex satisfying
$$\gamma(v)=0, \ \ \text{val}(v)=2, \ \ 
\mathbf{s}(v)=0,$$
the stability
condition 
$$2\gamma(v)-2 + \text{val}(v) + \epsilon\cdot \mathbf{s}(v) >0$$
holds for every strictly positive $\epsilon\in \mathbb{Q}$.
The valence of $v$, as before, counts all incident edges and markings.

\subsubsection{Mixed pointed spaces}

The $\mathbf T$-fixed loci for the stable quotients are described in terms of mixed pointed spaces. 
Let $s=(s_1,\ldots, s_r)$ be a tuple of non-negative integers.
Let $\overline{M}_{g,A|s}$ be the moduli space of 
genus $g$ curves with markings
$$
\{p_1,\ldots,p_A\} \
\cup\
\bigcup_{j=1}^r \{\widehat{p}_{j1},\ldots,\widehat{p}_{js_j}\}\ 
\in C^{ns}\subset C$$
satisfying the conditions
\begin{enumerate}
\item[(i)]
the points $p_i$ are distinct,
\item[(ii)]
the points $\widehat{p}_{jk}$ are distinct from the points $p_i$,
\end{enumerate}
with stability given by the ampleness
of $$\omega_C(\sum_{i=1}^A p_i + 
\epsilon \sum_{j,k} \widehat{p}_{jk})$$
for every strictly positive $\epsilon \in {\mathbb{Q}}$.
The conditions allow the points $\widehat{p}_{jk}$ and $\widehat{p}_{j'k'}$ to coincide.
If
 $$\mathbf{s} = \sum_{j=1}^r s_j,$$
then 
$\overline{M}_{g,A|s}= \overline{M}_{g,A|\mathbf{s}}$ 
defined in Section \ref{xcx}.

\subsubsection {Torus fixed quotients} 

Fix a decorated graph $(\Gamma, \nu, \gamma, s, \epsilon, \delta, \mu)$ indexing a $\mathbf{T}$-fixed locus of the 
moduli space $\overline Q_{g,m}(\mathbb G(r,n), d)$. 
The corresponding $\mathbf{T}$-fixed locus is, up to automorphisms, 
the product of mixed pointed spaces 
$$\prod_{v\in V}\overline M_{\gamma(v), \,\text{val}(v) | s(v)}.$$

The corresponding $\mathbf{T}$-fixed stable
quotients can be described explicitly. 
For each vertex $v$ of the graph, pick 
a curve $C_v$ in the mixed moduli space with 
markings $$\{p_1, \ldots, p_{\text{val}(v)}\} \cup \bigcup_{j=1}^{r} \{\widehat p_{j1}, \ldots, \widehat p_{js_j(v)}\}.$$ 
For each edge $e$, pick a rational curve $C_e$.
A pointed curve $C$ is obtained by gluing the curves $C_v$  and $C_e$ 
via the graph incidences, and distributing the markings on the 
domain via the assignment $\mu$. 
\begin {itemize}
\item [(i)] On the component $C_v$, the stable quotient is given by the exact sequence 
$$0\to \oplus_{j=1}^{r} 
\oh_{C_v}(-\sum_{k=1}^{s_j(v)}\widehat p_{jk})
\to \mathbb C^n\otimes \oh_{C_v} \to Q\to 0.$$ 
The first inclusion is the composition of 
$$\oplus_{j=1}^{r} \oh_{C_v}(-\sum_{k=1}^{s_j(v)}\widehat p_{jk})
\to \mathbb C^r\otimes \oh_{C_v}$$ with the 
$r$-plane $\mathbb C^r\otimes \oh_{C_v} \to \mathbb C^n\otimes \oh_{C_v}$ 
determined by $i_s$. 
\item [(ii)] For each edge $e$, consider the degree $\delta_e$ covering of the $\mathbf T$-invariant curve $\epsilon(e)$ in the Grassmannian $\mathbb G(r,n)$: $$f_e: C_e\to \epsilon(e)$$ ramified over the two torus fixed points. The stable quotient is obtained pulling back the tautological sequence of $\mathbb G(r,n)$ to $C_e$.
\end {itemize}
The gluing of stable quotients on different components is made possible by the compatibility of $i_s$ and $\nu$. 

\subsection{Contributions}
\label{fff3}
\subsubsection{Vertices for stable maps}
\label{fff31} 
Consider the case of a nondegenerate vertex $v$ occurring in
a graph for stable maps.
The vertex corresponds to the moduli space{\footnote{As usual we order all issues
and quotient by the overcounting.}}
$\overline{M}_{\gamma(v),\text{val}(v)}$. The vertex contribution is computed in \cite {GP} 
\begin{equation}\label{gtte}
\mathsf{Cont}(v)= \frac{{\mathsf{e}(\mathbb{E}^* \otimes T_{\nu(v)})}}{\mathsf{e}(T_{\nu(v)})} \frac{1}
{\prod_{e} \frac{w(e)}{\delta(e)}-\psi_e}.
\end{equation}
Here, ${\mathsf{e}}$ denotes the Euler class
$T_{\nu(v)}$ is the $\mathbf{T}$-representation on the tangent
space of $\mathbb G(r,n)$ at $\nu(v)$, and
$$\mathbb{E} \rarr \overline{M}_{\gamma(v),\text{val}(v)},$$
is the Hodge bundle.
Finally, the product in the denominator is over all half-edges
incident to $v$. The factor $w(e)$ denotes the $\mathbf{T}$-weight
of the tangent representation along the corresponding
$\mathbf{T}$-fixed edge, and $\psi_e$ denotes the cotangent
line at the corresponding marking of $\overline{M}_{\gamma(v),\text{val}(v)}$.

\subsubsection{Vertices for stable quotients} \label{fff32}
Next, let $v$ be a nondegenerate
vertex occurring in a graph for stable quotients.
For simplicity, assume  
$$\iota_s(j)=j, \ \ 1\leq j \leq r.$$
The vertex corresponds to the moduli space{\footnote{Again, we order all
issues and quotient by the overcounting.}}
$\overline{M}_{\gamma(v), \text{val}(v)|s(v)}$ where the subsheaf is given by 
$$0 \rarr S = \oplus_{j=1}^r \oh_C(-\sum_{k=1}^{s_j(v)}\widehat{p}_{jk}) 
\stackrel{\iota_s}{\rarr} \com^n \otimes \oh_C\rarr Q \rarr 0.$$
The vertex contribution, determined by the {\it moving} part of $RHom(S,Q)$, is
\begin{equation} \label{frt}
\mathsf{Cont}(v)= \frac{{\mathsf{e}(\text{Ext}^1(S,Q)^{\text {m}})}}{\mathsf{e}(\text{Ext}^0(S,Q)^{\text{m}})} \frac{1}
{\prod_{e} \frac{w(e)}{\delta(e)}-\psi_e} \ \ .
\end{equation} Since the $\text{Ext}$ spaces are not separately of constant rank, a better
form is needed for \eqref{frt}. 

Let $S_i\subset C$ be the divisor associated to points
corresponding to $s_i$.
By the results of Section \ref{txtq}, we see that \eqref{frt} is equivalent to 
\begin{eqnarray*}
\mathsf{Cont}(v)&=& 
\frac{{\mathsf{e}(\mathbb{E}^* \otimes T_{\nu(v)})}}{\mathsf{e}(T_{\nu(v)})} \frac{1}
{\prod_{e} \frac{w(e)}{\delta(e)}-\psi_e} \cdot \\
& & \frac{1}{\prod_{i\neq j} \mathsf{e}(H^0(\oh_C(S_i)|_{S_j})\otimes [\mathsf{w}_j-\mathsf{w}_i])} \cdot \\
& & 
\frac{1}{\prod_{i, j^*} \mathsf{e}(H^0(\oh_C(S_i)|_{S_i})\otimes
[\mathsf{w}_{j^*}-\mathsf{w}_i]))} \  \
\end{eqnarray*}
where the products in the last factors satisfy the following conditions
$$1\leq i \leq r, \ \ 1\leq j \leq r, \ \ r+1 \leq j^* \leq n\ \ .$$ The brackets $\left[\,\cdot\,\right]$ in the above expression denote the trivial line bundle with the specified weights.

While the vertex contributions for stable maps and stable quotients appear
quite different, the genus dependent part of the integrand
involving the Hodge bundle is the {\em same}. The differences all
involve the local geometry of the points.

\subsection{Matching}
Under the map to
$\overline{Q}_{g,m}({\mathbb{G}}(1,\binom{n}{r}),d)$,
the stable map side has
many genus 0 tails which are collapsed. Similarly, the
stable quotient side has many splitting types of
the subbundle $S$ which are collapsed. 
The differences in the localization formulas occur
entirely in the nondegenerate vertices. 
For noncollapsed edges (not occurring in
genus 0 tails of the stable map space) and
noncollapsed degenerate vertices of valence 2, the edge 
and vertex contributions
exactly coincide.

The
crucial step in the argument is to notice
Theorems 3 and 4 are a consequence of a universal
calculation in a moduli space of pointed curves.
In fact, the universal calculation is genus independent since the
genus dependent integrand factors match.

In genus 0, a geometric argument can be given.
Since
$$\overline{M}_{0,m}(\mathbb{G}(r,n),d)  \ \ \text{and} \ \
\overline{Q}_{0,m}(\mathbb{G}(r,n),d)$$
are nonsingular of expected dimension,
the virtual class in both cases is the usual fundamental class.
Moreover, since the moduli space are
irreducible{\footnote{See \cite{BP, Jes}.}} and birational,
Theorem 4 in the form
\begin{equation}\label{unn7}
c_*\iota_{M*}\Big(
[\overline{M}_{0,m}({\mathbb{G}}(r,n),d)]^{vir}\Big)= \\
  \iota_{Q*}\Big(
[\overline{Q}_{0,m}({\mathbb{G}}(r,n),d)]^{vir}\Big).
\end{equation}
is  trivial. The image of
$c \circ \iota_M$ simply coincides with the image of $\iota_Q$.

 Let $\xi\in \mathbb{G}(r,n)^{\mathbf{T}}$ be a fixed point
and let 
$$\iota(\xi) \in {\mathbb{G}}(1, \binom{n}{r})$$
be the image in the Pl\"ucker embedding.
Consider the $\mathbf{T}$-fixed locus of 
$\overline{Q}_{0,m}(\mathbb{G}(1, \binom{n}{r}),d)$
which corresponds to a single vertex over $\iota(\xi)$
with no edges.
By the discussion of Section \ref{fff2}, the associated
$\mathbf{T}$-fixed locus is
$\overline{M}_{0,m|d}/{{\mathbb{S}}_d}$.

Equality \eqref{unn7} implies a matching after $\mathbf{T}$-equivariant
localization. In particular, there is a matching obtained
for $\mathbf{T}$-equivariant residues on the locus 
$\overline{M}_{0,m|d}/{{\mathbb{S}}_d}$ over $\iota(\xi)$.
The residue of
$$c_*\iota_{M*}\Big(
[\overline{M}_{0,m}({\mathbb{G}}(r,n),d)]^{vir}\Big)$$
is a graph sum over all $\mathbf{T}$-fixed
point loci of $\overline{M}_{0,m}({\mathbb{G}}(r,n),d)$
which contract to $\overline{M}_{0,m|d}/{{\mathbb{S}}_d}$
over $\iota(\xi)$.
Similarly,
the residue of
$$\iota_{Q*}\Big(
[\overline{Q}_{0,m}({\mathbb{G}}(r,n),d)]^{vir}\Big)$$
is a splitting sum over all $\mathbf{T}$-fixed
point loci of $\overline{Q}_{0,m}({\mathbb{G}}(r,n),d)$
which collapse to $\overline{M}_{0,m|d}/{{\mathbb{S}}_d}$
over $\iota(\xi)$.

While the equality of residues holds on
$\overline{M}_{0,m|d}/{{\mathbb{S}}_d}$, we can canonically lift
both sides to symmetric polynomials in $\widehat{\psi}_j$
and $D_J$ on
$\overline{M}_{0,m|d}$. On the left side, we use the contribution
formulas of Section \ref{fff31} and Lemma \ref{nbv} to obtain
$$L_{d,\xi}(\widehat{\psi}_j, D_J).$$
On the right side we use the contribution formulas of Section \ref{fff32}
to obtain
$$R_{d,\xi}(\widehat{\psi}_j, D_J).$$
The symmetry of $L_{d,\xi}$ and $R_{d,\xi}$
is with respect the points $\widehat{p}_1,\ldots, \widehat{p}_d$.
We may take the symmetric polynomials $L_{d,\xi}$ and 
$R_{d,\xi}$ to be in the canonical form of Section \ref{can77}.
The polynomials $L^C_{d,\xi}$ and $R^C_{d,\xi}$
are independent of $m$.

We know the push-forwards of $L^C_{d,\xi}$ and $R^C_{d,\xi}$
to $\overline{M}_{0,m|d}/{\mathbb{S}_d}$ are equal by the
matching of residues in genus 0.
By the independence result of Section \ref{indd},
we  conclude the much stronger equality
\begin{equation}\label{strr}
L^C_{d,\xi} =R^C_{d,\xi}
\end{equation}
as abstract {\em polynomials}.

The  equality \eqref{strr} as polynomials
 implies Theorems 3 and 4
for arbitrary genus since the cotangent calculus is
genus independent. \qed

\subsection{Independence} \label{indd}
\subsubsection{Polynomials} \label{nn59}
Consider variables $\widehat{\psi}_1, \ldots, \widehat{\psi}_d$
and 
$$\{\ D_J\ |\ J \subset \{ 1,\ldots, d\}, \ |J| \geq 2 \ \}$$
for fixed $d\geq 0$.
Given a polynomial 
$P(\widehat{\psi}_j, D_J)$, we obtain a canonical form $P^C$
in the sense of Section \ref{can77}.

We view $P^C$ in two different ways.
First, 
$P^C$
yields a class 
\begin{equation}\label{pw342}
P^C = P(\widehat{\psi}_j, D_J) \in 
A^*(\overline{M}_{0,m|d},\mathbb{Q})
\end{equation}
for every $m$. We will always take $m\geq 3$ 
to avoid unstable cases.
Second, $P^C$ is
an abstract polynomial.
If $P^C$ always vanishes in the
first sense \eqref{pw342}, then we will show that $P^C$
vanishes as an abstract polynomial.

If $P(\widehat{\psi}_j, D_J)$ is symmetric with
respect to the $\mathbb{S}_d$-action on the variables,
then $P^C$ is also symmetric. The class \eqref{pw342}
lies in the $\mathbb{S}_d$-invariant sector,
\begin{equation}\label{pw3421}
P^C \in 
A^*(\overline{M}_{0,m|d},\mathbb{Q})^{\mathbb{S}_d} =
A^*(\overline{M}_{0,m|d}/\mathbb{S}_d,\mathbb{Q}).
\end{equation}
Hence, for symmetric $P$, 
only the vanishing
in  $A^*(\overline{M}_{0,m|d}/\mathbb{S}_d,\mathbb{Q})$
will be required to show $P^C$ vanishes as
an abstract polynomial.

\subsubsection{Partitions}
Fix a codimension $k$. 
Let 
$$\mathcal{P}=(\mathcal{P}_1, \ldots, \mathcal{P}_\ell)$$
 be a set partition of $\{1,\ldots, d\}$
with $\ell \geq d-k$ nonempty parts,
$$\cup_{i=1}^\ell {\mathcal{P}}_i = \{1,\ldots,d\}.$$
The parts $\mathcal{P}_i$ are ordered by lexicographical
ordering.{\footnote{The choice of
ordering will not play a role
in the argument.}} 
Let 
$$\tau=(t_1,\ldots, t_\ell), \ \ \ \sum_{i=1}^\ell t_i = k-d+\ell$$
be an ordered  partition of $k-d+\ell$. 
The parts $t_i$ {\em are} allowed to be 0.

Let ${\mathsf{P}}[d,k]$ be the
set of all such pairs $[\mathcal{P},\tau]$.
Let $\mathbf{V}[d,k]$
be a $\mathbb{Q}$-vector space with basis 
given by the elements of
${\mathsf{P}}[d,k]$.
To each pair
$[\mathcal{P},\tau] \in \mathsf{P}[d,k]$ and integer $m\geq 3$, we
 associate 
the class
$$X_m[\mathcal{P},\tau]=\widehat{\psi}_{\mathcal{P}_1}^{\hpo t_1}
 \ldots \widehat{\psi}_{\mathcal{P}_\ell}^{\hpo t_\ell} 
\cdot D_{\mathcal{P}_1} \cdots D_{\mathcal{P}_\ell}
\ \  
\in A^{k}
(\overline{M}_{0,m|d})\ .$$

\subsubsection{Pairing}

Let $\mathbf{W}_m$ the  $\mathbb{Q}$-vector space 
with basis given by the symbols
$[\mathcal{S},\mu]$
where $\mathcal{S}\subset \overline{M}_{0,m|d}$ is a stratum
of dimension $s\geq k$ and
$\mu$ is a monomial in the variables 
$$\psi_1,\ldots, \psi_m$$
of degree $s-k$.

There is a canonical Poincar\'e pairing
$$I: 
{\mathbf{V}}[d,k] \times
\mathbf{W}_{m}
\rarr \mathbb{Q}$$
defined on the bases by
$$I([\mathcal{P},\tau],[\mathcal{S},\mu]) = 
\int_{\mathcal{S}}
X_m[\mathcal{P},\tau]
\cup \mu(\psi_{1},\ldots,\psi_{m}) 
.$$

\begin{Lemma} 
\label{lm23}
A vector $v\in \mathbf{V}[d,k]$
is 0 if and only if $v$ is in the null space of all
the pairings $I$ for $m\geq 3$.
\end{Lemma}

\bpf 
To $[\mathcal{P},\tau]\in \mathsf{P}[d,k]$, 
we associate  $Y_m[\mathcal{P},\tau]\in \mathbf{W}_m$, 
where
$$m= 4+ k-d+2\ell$$
and
$\ell$ is the length of $\mathcal{P}$, by the following
construction.

Let $\mathcal{S}\subset \overline{M}_{0,m|d}$ be the stratum consisting
of a chain of $\ell+2$ rational curves attached head to
tail,
$$R_0-R_1-R_2 -\ldots - R_{\ell}-R_{\ell+1},$$
with the $m$ markings $p_1,\ldots, p_m$ distributed by the rules:
\begin{enumerate}
\item[(i)] $R_0$ and $R_{\ell+1}$ each carry exactly 2 markings,
\item[(ii)] For $1 \leq i \leq {\ell},$ $R_i$ carries $t_i+1$ markings,
\item[(iii)] the marking are distributed in order from left to right.
\end{enumerate}
We write $r_i$ for the minimal label of the markings on $R_i$. 
The $d$ markings $\widehat{p}_1,\ldots, \widehat{p}_d$ are
distributed by the rules
 \begin{enumerate}
\item[(iv)] $R_0$ and $R_{\ell+1}$ each carry 0 markings
\item[(v)] For $1 \leq i \leq {\ell},$ $R_i$ carries the markings corresponding to $\mathcal{P}_i$.
\end{enumerate}
The dimension of $\mathcal{S}$ is easily calculated, 
\begin{eqnarray*}
\text{dim}(\mathcal{S})& =& \text{dim}(\overline{M}_{0,m|d})- \ell-1 \\
& = & 4+k-d+2\ell+d-3-\ell-1 \\
& =& \ell+k.
\end{eqnarray*}
The associated element of $\mathbf{W}_m$ is defined by
$$Y_m[\mathcal{P},\tau] = [\mathcal{S},\psi_{r_1} \cdots
 \psi_{r_{\ell}}] .$$

The next step is to find when the pairing 
\begin{equation}\label{gt12}
I([\mathcal{P},\tau], Y_{m'}[\mathcal{P}',\tau'])
\end{equation}
is nontrivial for 
$$[\mathcal{P},\tau], \ [\mathcal{P}',\tau']\in \mathsf{P}[d,k],$$
with $m'=4+k-d+2\ell'$.
By definition, the pairing \eqref{gt12} equals
\begin{multline}\label{nt56}
\int_{\mathcal{S}'}
X_{m'}[\mathcal{P},\tau]
\cup
\psi_{r_1'} \cdots
 \psi_{r'_{\ell'}} =\\ 
\int_{\mathcal{S}'}\psi_{r_1'} \cdots
 \psi_{r'_{\ell'}}\cdot
\widehat{\psi}_{\mathcal{P}_1}^{\hpo t_1}
 \ldots \widehat{\psi}_{\mathcal{P}_\ell}^{\hpo t_\ell} 
\cdot D_{\mathcal{P}_1} \cdots D_{\mathcal{P}_\ell}
\end{multline}
The integral \eqref{nt56} is  calculated by distributing
the diagonal points corresponding to $D_{\mathcal{P}_j}$
to the components $R'_i$ of curves in $\mathcal{S}'$
in all possible ways.
Note that unless there is at least one diagonal $D_{\mathcal{P}_j}$ distributed to
{\em each} $R'_i$ for $1\leq i \leq \ell'$, the 
contribution to the integral \eqref{nt56} vanishes.
Hence, nonvanishing implies $\ell \geq \ell'$.

If $\ell = \ell'$, then the distribution rule (v)
implies
the set theoretic intersection
$$\mathcal{S}' \cap  D_{\mathcal{P}_1} \cap \cdots \cap  D_{\mathcal{P}_\ell}$$
is empty unless $\mathcal{P}=\mathcal{P}'$.
If $\mathcal{P}=\mathcal{P}'$, the only nonvanishing
diagonal distribution is given by sending
$D_{\mathcal{P}_i}$ to $R_i'$.
The integral \eqref{nt56} is easily seen to be nonzero
then if and only if $\tau=\tau'$. Indeed, the contribution of $R'_i$ to the integral is $$\int_{\overline M_{0, t'_i+3\large {|} |{\mathcal P_i}|}} \psi_{r'_i}\, \widehat \psi_{\mathcal P_i}^{t_i}\cdot D_{\mathcal P_i}=\int_{\overline M_{0, t'_i+4}} \psi_{r'_i}\, \psi^{t_i}_{t'_i+4}=\begin {cases} t_i+1 &\text{ if } t_i= t'_i\\ 0 & \text {otherwise}\end {cases}.$$ 

The linear functions on $\mathbf{V}[d,k]$
determined by $I(\cdot, Y_{m'}[\mathcal{P}',\tau'])$
are block lower-triangular with respect to the
partial ordering by the length of the set partition.
Moreover, the diagonal blocks are themselves diagonal 
with nonzero entries.
\epf

Following the notation of Section \ref{nn59},
Lemma \ref{lm23} proves that if 
$$P^C \in A^*(\overline{M}_{0,m|d}, \mathbb{Q})$$
vanishes for all $m\geq 3$, then
$P^C$ vanishes as an abstract polynomial.
The proofs of Theorems 3 and 4 are therefore
complete.

\section{Tautological relations}\label{ttaa6}

\subsection{Tautological classes}
Let $g \geq 2$.
The tautological ring of the moduli space of curves
$$R^*(M_g) \subset A^*(M_g,\mathbb{Q})$$
is generated by the classes
$$\kappa_i = \epsilon_*(\psi_1^{i+1}), \ \ \ \ M_{g,1} \stackrel{\epsilon}{\rarr}M_{g}.$$
Here, $\kappa_0=2g-2$ is a multiple of the unit class.
A conjectural description of $R^*(M_g)$ is presented in 
\cite{Faber}.
The basic vanishing result, 
$$R^{i}(M_g)= 0$$
for $i>g-2$, 
has been proven by Looijenga \cite{L}.

\subsection{Relations}
Let $g \geq 2$ and $d\geq 0$.
The moduli space 
$$M_{g,0|d}\stackrel{\epsilon}{\rarr} M_g $$
is simply the $d$-fold product of the
universal curve over $M_g$. 
Given an element
$$[C, \widehat{p}_1,
\ldots,\widehat{p}_d] \in {M}_{g,0|d}\ , $$
there is a canonically associated stable quotient
\begin{equation}\label{jwq2}
0 \rarr \oh_C(-\sum_{j=1}^d \widehat{p}_j) \rarr \oh_C \rarr Q \rarr 0.
\end{equation}
Consider the universal curve
$$\pi: U \rarr {M}_{g,0|d}$$
with universal quotient sequence
$$0 \rarr S_U \rarr \oh_U \rarr Q_U \rarr 0$$
obtained from \eqref{jwq2}.
Let
$$\mathbb{F}_d= -R\pi_*(S^*_U) \in K({M}_{g,0|d})$$
be the class in $K$-theory.
For example,
$$\mathbb{F}_0 = \mathbb{E}^*-\com$$
is the dual of the Hodge bundle minus a rank 1 trivial
bundle. 

By Riemann-Roch,
the rank of 
$\mathbb{F}_d$ is 
$${r}_g(d)=g-d-1.$$ 
However, $\mathbb{F}_d$ is not always represented by a bundle. 
By the derivation of Section \ref{txtq},
\begin{equation}\label{laloo1}
\mathbb{F}_d = \mathbb{E}^*- \mathbb{B}_d - \com,
\end{equation} where 
$\mathbb{B}_d$ has fiber
$H^0(C,\oh_C(\sum_{j=1}^d \widehat{p}_j)|_{\sum_{j=1}^d \widehat{p}_j})$
over $[C, \widehat{p}_1,\ldots,\widehat{p}_d].$
Alternatively, $\mathbb{B}_d$ is the $\epsilon$-relative
tangent bundle.

\begin{Theorem}
For every integer $k > 0$, 
$$\epsilon_* \left( c_{r_g(d)+2k}(\mathbb{F}_d) \right) = 0  \  \in R^*(M_g).$$
\end{Theorem}

Since the morphism 
$\epsilon$ has fibers of dimension $d$,
$$\epsilon_* \left( c_{r_g(d)+2k}(\mathbb{F}_d) \right) \in R^{g-2d-1+2k}(M_g).$$
By Looijenga's vanishing, Theorem 5 is only nontrivial when
$$0 \leq {2d-2k -1} \leq g-2.$$ The vanishing of Theorem 5
 does not naively extend. We calculate 
\begin{equation}
\label{heww1}
\epsilon_* \left( c_{r_g(1)+1}(\mathbb{F}_1) \right) = 
\kappa_{g-2}-\lambda_1 \kappa_{g-3}+ \ldots + (-1)^{g-2}\kappa_0\lambda_{g-2} 
\end{equation}
in $R^{g-2}(M_g)$
by \eqref{laloo1}. However, the class \eqref{heww1}
is known not to vanish by the pairing with $\lambda_g \lambda_{g-1}$
calculated in \cite{P}. 

Theorem 5 directly yields relations among the generators
$\kappa_i$
of 
$R^*(M_g)$
by the standard $\epsilon$ push-forward rules \cite {Faber}.
The construction is more subtle than the
method of \cite{Faber} as the relations 
only hold after push-forward.
An advantage is that the boundary terms of
the relations here can easily be calculated.

\subsection{Example}\label{ggrr2}
The Chern classes of $\mathbb F_d$ can be easily computed. 
Recall the divisor $D_{i,j}$ where the markings $\widehat p_i$ 
and $\widehat p_j$ coincide. Set $$\Delta_i=D_{1,i}+\ldots+D_{i-1,i},$$ 
with the convention $\Delta_1=0.$ 
Over $[C, \widehat p_1, \ldots, \widehat p_d],$ the virtual 
bundle $\mathbb F_d$ is the formal difference 
$$H^1(\mathcal O_C(\widehat p_1+\ldots+\widehat p_d))
-H^0(\mathcal O_C(\widehat p_1+\ldots+\widehat p_d)).$$ 
Taking the cohomology of the exact sequence 
$$0\to \mathcal O_C(\widehat p_1+\ldots+\widehat p_{d-1})\to 
\mathcal O_C(\widehat p_1+\ldots+\widehat p_d)\to  
\mathcal O_C(\widehat p_1+\ldots+\widehat p_d)|_{\widehat p_d}\to 0,$$ 
we find $$c(\mathbb F_d)=
\frac{c(\mathbb F_{d-1})}{1+\Delta_d-\widehat \psi_d}.$$ 
Inductively, we obtain
 \begin{equation}\label{laloo}c(\mathbb F_d)=
\frac{c(\mathbb E^*)}{(1+\Delta_1-\widehat{\psi}_1)\cdots 
(1+\Delta_d-\widehat {\psi}_d)}.\end{equation}

In the $d=2$ and $k=1$ case, Theorem $5$ gives the vanishing of the 
class $$\epsilon_*\, c_{g-1}(\mathbb F_d)=\epsilon_* 
\left[\frac{c(\mathbb E^*)}{(1-\widehat\psi_1)(1+\Delta-\widehat \psi_{2})}\right]^{g-1},$$ where $\Delta$ is the divisor of coincident markings on $M_{g, 0|2}$. 
The superscript indicates the degree $g-1$ part of the bracketed expression. 
Expanding, we obtain \begin{equation}\label{sq}\sum_{i} (-1)^i\lambda_{g-1-i} \sum_{i_1+i_2=i} 
\epsilon_* \left(\widehat \psi_1^{i_1}\, (\widehat\psi_2-\Delta)^{i_2}\right)=0.\end{equation} 
We have $$\epsilon_{*}
\left (\widehat \psi_1^{i_1}\, (\widehat \psi_2-\Delta)^{i_2}\right)=
\sum_{m}  (-1)^{i_2-m} \binom{i_2}{m}\epsilon_{*} 
\left(\widehat \psi_1^{i_1} \,\widehat \psi_2^m\, \Delta^{i_2-m}\right).$$ 
Using $$\Delta^2=-\widehat \psi_1 \,\Delta = -\widehat \psi_2\, \Delta$$ 
and the $\epsilon$-calculus rules in \cite {Faber}, we rewrite the last expression as 
$$-\sum_{m\neq i_2} \binom{i_2}{m} \epsilon_{*} 
(\widehat \psi_1^{i_1+i_2-1} \Delta)+\epsilon_{*} (\widehat \psi_1^{i_1}\, \widehat \psi_2^{i_2})
=-(2^{i_2}-1) \kappa_{i_1+i_2-2}+\kappa_{i_1-1}\,\kappa_{i_2-1}.$$ 
After summing over 
$i_1,i_2$ in \eqref{sq}, we arrive at the relation 
\begin{equation}\label{htyy}
\sum_{i=2}^{g-1} (-1)^i \lambda_{g-1-i} \left(\left(\sum_{i_1+i_2=i} \kappa_{i_1-1} \kappa_{i_2-1}\right) 
-(2^{i+1}-i-2)\kappa_{i-2}\right)=0
\end{equation} 
in $R^{g-3}(M_g)$. 

The $\lambda$ classes
can be expressed in terms of the $\kappa$ classes by Mumford's Chern
character calculation
$$\text{ch}_{2\ell}(\mathbb{E}) = 0, \ \ \ \ 
\text{ch}_{2\ell-1}(\mathbb{E}) = \frac{B_{2\ell}}{(2\ell)!}
\kappa_{2\ell-1}$$
for $\ell>0$.
From \eqref{htyy}, we
obtain a relation 
involving only the tautological generators $\kappa_i$.
To illustrate, in genus $6$, we obtain the relation 
$$25\kappa_1^3+15 912 \kappa_3 - 1 080 \kappa_1 \kappa_2= 0,$$ 
which is consistent with the presentation of $R^*(M_6)$ 
in \cite {Faber}.

\subsection{Brill-Noether construction}
The $k=1$ case of Theorem 5 for positive $d\leq g$
admits an alternative derivation
via Brill-Noether theory.{\footnote{The Brill-Noether connection
was suggested by C. Faber who recognized equation \eqref{htyy}.}}

To start, consider the rank $d$ bundle,
$$\mathbb{W}_d\rarr M_{g,0|d}\ ,$$ 
with fiber
$H^0(C,\omega_C|_{\sum_{j=1}^d \widehat{p}_j})$
over $[C, \widehat{p}_1,\ldots,\widehat{p}_d].$
There is a canonical map of vector bundles on
$M_{g,0|d}$,
$$\rho:\mathbb{E} \ {\rarr}\ \mathbb{W}_d\ ,$$
defined by the restriction 
$H^0(C,\omega_C) \rarr H^0(C,\omega_C|_{\sum_{j=1}^d \widehat{p}_j})$.
After dualizing, we obtain
$$\rho^*:\mathbb{W}_d^*\  {\rarr}\ \mathbb{E}^*\ .$$
If $\rho^*$ fails to have maximal rank at
$[C,\widehat{p}_1,\ldots,\widehat{p}_d]\in M_{g,0|d}$ , then
the divisor $\widehat{p}_1+\ldots +\widehat{p}_d$
must move in a nontrivial linear series. 
The degeneracy locus of $\rho^*$ precisely
defines the Brill-Noether variety \cite{ACGH} 
$$G^1_d\subset M_{g,0|d}\ ,$$
well-known to be of expected codimension $g-d+1$.
Since 
$$\epsilon: G^1_d \rarr M_g$$
has positive dimensional fibers, certainly
\begin{equation*}
\epsilon_*[G^1_d]=0 \ \in A^*(M_g)
\end{equation*}
By the Porteous formula \cite{Ful},
$$[G^1_d] = c_{g-d+1}(\mathbb{E}^*- \mathbb{W}^*_d)\ .$$
Hence,  we obtain the relation 
\begin{equation}\label{kq12}
\epsilon_*\left( c_{g-d+1}(\mathbb{E}^*- \mathbb{W}^*_d)\right)=0 \ \in R^*(M_g) \ .
\end{equation}

\begin{Lemma}\label{p67} $\mathbb{W}_d \stackrel{\sim}{=} \mathbb{B}^*_d\ $. 
\end{Lemma}

\bpf 
Let $\widehat{P}\subset C$ denote the divisor $\widehat{p}_1+\ldots +\widehat{p}_d$.
The fiber of $\mathbb{W}_d$ over $[C, \widehat{p}_1,\ldots,\widehat{p}_d]$
is 
$$\Ext^0(\oh_C,\omega_C|_{\widehat{P}}) \stackrel{\sim}{=} \Ext^1(\oh_{\widehat{P}},\oh_C)^*$$
by Serre duality. Let
$$I = [\oh_C(-\widehat{P}) \rarr \oh_C]$$
denote  the complex of line bundles in grade -1 and 0.
Since $I$ is quasi-isomorphic to $\oh_{\widehat{P}}$, we find
$$\Ext^{1}(I,\oh_C) \stackrel{\sim}{=} \Ext^1(\oh_{\widehat{P}},\oh_C)$$
On the  other hand, we have
$$I^* = [\oh_C \rarr \oh_C(\widehat{P})] \ \ {\text{and}} \ \
\Ext^{1}(\oh_C,I^*) \stackrel{\sim}{=} \Ext^0(\oh_C,\oh_{\widehat{P}}(\widehat{P})).$$
We have hence found a canonical isomorphism
$$\Ext^1(\oh_{\widehat{P}},\oh_C)\stackrel{\sim}{=} \Ext^0(\oh_C,\oh_{\widehat{P}}(\widehat{P}))$$
where the latter space is the fiber of $\mathbb{B}_d$
\epf


The $k=1$ case of Theorem 5 concerns the class
\begin{eqnarray*}
c_{g-d+1}(\mathbb{F}_d) & = & 
c_{g-d+1}(\mathbb{E}^* - \mathbb{B}_d - \com ) \\
& = & c_{g-d+1}(\mathbb{E}^* - \mathbb{B}_d) \\
& = & c_{g-d+1}(\mathbb{E}^* - \mathbb{W}^*_d).
\end{eqnarray*}
Hence, the vanishing 
$$\epsilon_*( c_{g-d+1}(\mathbb{F}_d)) = 0$$
of Theorem 5 exactly coincides with the Brill-Noether
vanishing \eqref{kq12}.

Theorem 5 may be viewed as a generalization of
Brill-Noether vanishing obtained from the
virtual geometry of the moduli of stable quotients.

\subsection{Proof of Theorem 5}
Consider the proper morphism
$$\nu: Q_{g}(\proj^1,d) \rarr M_g.$$
The universal curve 
$$\pi:U \rarr Q_{g}(\proj^1,d)$$ carries the  basic
divisor classes
$${s} = c_1(S_U^*), \ \ \ \  \omega= c_1(\omega_\pi)$$
obtained from the universal subsheaf $S_U$ and the $\pi$-relative dualizing
sheaf. 
The
class
\begin{equation}\label{p236}
\nu_*\left( \pi_*(s^a \omega^b)\cdot 0^c \cap [Q_g(\proj^1,d)]^{vir} \right)
\in A^*(M_g,\mathbb{Q}),
\end{equation}
where $0$ is first Chern class of the trivial bundle,
certainly vanishes if $c>0$. 
Theorem 5 is proven by calculating \eqref{p236} by localization.
We will find Theorem 5 is a subset of a richer
family of relations.

Let the 1-dimensional torus $\com^*$ act on a 2-dimensional vector
space $V\stackrel{\sim}{=}\com^2$ with diagonal weights
$[0,1]$.  The $\com^*$-action lifts
canonically to the following spaces and sheaves:
$$\proj(V),\ \
Q_{g}(\proj(V),d), \ \ U, \ \ S_U, \ \ \text{and}\ \ \omega_\pi.$$
We lift the $\com^*$-action to a rank 1 trivial bundle on $Q_g(\proj(V), d)$
by specifying fiber weight $1$.
The choices determine a $\com^*$-lift of the
class 
$$\pi_*(s^a\cdot \omega^b)
\cdot 0^c \cap [Q_g(\proj(V),d)]^{vir}\in 
A_{2d+2g-1-a-b-c}(
Q_g(\proj(V),d),\mathbb{Q}).$$
 
The push-forward  \eqref{p236} is determined by 
the virtual localization formula \cite{GP}. There are
only two $\com^*$-fixed loci.
The first corresponds to a vertex lying over $0\in \proj(V)$.
The locus is isomorphic to
$$M_{g,0|d}\ /\ \mathbb{S}_d$$
 and the
associated subsheaf \eqref{jwq2}
lies in the first factor of
$V \otimes  \oh_C$ when considered as a stable quotient in
the moduli space $Q_g(\proj(V),d)$.
Similarly, the second fixed locus 
corresponds to a vertex lying over $\infty\in \proj(V)$.

The localization contribution of the first locus to \eqref{p236} is 
$$\frac{1}{d!}\epsilon_*
\left(\pi_*(s^a\omega^b)\cdot c_{g-d-1+c}(\mathbb{F}_d)\right)$$
where $s$ and $\omega$ are the corresponding classes on
the universal curve over $M_{g,0|d}$.
Let $c_-(\mathbb{F}_d)$ denote the total Chern class of $\mathbb{F}_d$
evaluated at $-1$.
The localization contribution of the second locus is
$$\frac{(-1)^{g-d-1}}{d!}\epsilon_*\Big[
\pi_*\left((s-1)^a\omega^b\right)
\cdot c_{-}(\mathbb{F}_d)\Big]^{g-d-2+a+b+c}$$
where $[\gamma]^k$ is the part of $\gamma$ in $A^k(M_{g,{0|d}},\mathbb{Q})$.

Both localization contributions are found by straightforward
expansion of the vertex formulas of Section \ref{fff32}.
Summing the contributions yields the following result.

\begin{Proposition} \label{htht}
Let $c>0$. Then
\begin{multline*}
\epsilon_*\Big(
\pi_*(s^a\omega^b)\cdot c_{g-d-1+c}(\mathbb{F}_d) + \\
(-1)^{g-d-1} \Big[
\pi_*\left((s-1)^a\omega^b\right)
\cdot c_{-}(\mathbb{F}_d) \Big] ^{g-d-2+a+b+c}
\Big) = 0
\end{multline*}
in $R^*(M_g)$.
\end{Proposition}

If $a=0$ and $b=1$, the relation of Proposition \ref{htht}
specializes to Theorem 5 for even  $c=2k$.
\qed

\begin{Question}
Do the relations obtained from Proposition \ref{htht}
generate all the relations among the
classes $\kappa_i$ in $R^*(M_g)$ ?
\end{Question}

\subsection{Further examples}
Let $\sigma_i\in A^1(U,\mathbb{Q})$ be the
class of the $i^{th}$ section of the universal curve 
$$\pi: U \rarr M_{g,0|d}\ .$$ The class $s=c_1(S_U^*)$ 
of Proposition \ref{htht} is
$$s=\sigma_1 + \ldots +\sigma_d\ \in A^1(U,\mathbb{Q}).$$
We calculate 
\begin{eqnarray*}
\pi_*(s) & = & d \\
\pi_*(\omega)& = & 2g-2 \\
\pi_*(s\ \omega) & =&  \widehat{\psi}_1 +\ldots +\widehat{\psi}_d\\
\pi_*(s^2) & = & -(\widehat{\psi}_1 +\ldots +\widehat{\psi}_d) + 2 \Delta\\
\end{eqnarray*}
in  $A^*(M_{g,0|d},
\mathbb{Q})$,
where 
$$ \Delta = \sum_{i<j} D_{i,j} \ \in A^1(M_{g,0|d},
\mathbb{Q})$$
is the symmetric diagonal.
The push-forwards $\pi_*(s^a \omega^b)$ are all easily
obtained.

Using the above $\pi_*$ calculations,
the $a=1$, $b=1$, $c=2k$ case of Proposition \ref{htht}
yields 
$$\epsilon_*\Big( 2
(\widehat{\psi}_1 +\ldots +\widehat{\psi}_d)
\cdot c_{r_g(d)+2k}(\mathbb{F}_d)  + (2g-2)\
c_{r_g(d)+2k+1}(\mathbb{F}_d)\Big) =0.
$$
The $a=2$, $b=0$, $c=2k$ case yields
$$\epsilon_*\Big( -2
(\widehat{\psi}_1 +\ldots +\widehat{\psi}_d-2\Delta)
\cdot c_{r_g(d)+2k}(\mathbb{F}_d)  + 2d\cdot
c_{r_g(d)+2k+1}(\mathbb{F}_d)\Big) =0.
$$
Summation yields a third relation,
$$\epsilon_*\Big( 2\Delta
\cdot c_{r_g(d)+2k}(\mathbb{F}_d) +(d+g-1)\cdot
c_{r_g(d)+2k+1}(\mathbb{F}_d)\Big) =0.
$$
The relations of Proposition \ref{htht} include the
classes $c_{r_g(d)+2k+1}(\mathbb{F}_d)$
omitted in Theorem 5.

\section{Calabi-Yau geometry}
\label{lllll}
The moduli of stable quotients may be used to define
counting invariants in the local Calabi-Yau geometries.
For example consider the conifold, the total space of
$$\oh_{\proj^1}(-1) \oplus \oh_{\proj^1}(-1) \rarr \proj^1.$$
Just as in Gromov-Witten theory, we define
\begin{equation}\label{gt123}
 N_{g,d} = \frac{1}{d^2}\int_{[\overline{Q}_{g,2}(\proj^1,d)]^{vir}}
\mathsf{e}( R^1\pi_*(S_U) \oplus R^1\pi_*(S_U)) \cup \text{ev}_1^*(H)
\cdot \text{ev}_2^*(H)
\end{equation}
where $S_U$ is the universal subsheaf on the
universal curve 
$$\pi: U \rarr \overline{Q}_{g,2}(\proj^1,d)$$
and $H\in H^2(\proj^1,\mathbb{Q})$
is the hyperplane class.
The two point insertions are required for stability in genus 0.
Let
$$F(t) = \sum_{g\geq 1} N_{g,1} t^{2g}.$$

\begin{Proposition} The local invariants $N_{g,d}$
are determined by the following two equations,
$$N_{g,d} = {d^{2g-3}} N_{g,1},$$
$$F(t)= \left(\frac{t/2}{\sin(t/2)} \right)^2.$$
\end{Proposition}

\begin{proof} We compute the integral $N_{g,d}$
by localization. Let $\com^*$ act on 
the vector space
$V\stackrel{\sim}{=}\com^2$ with diagonal weights
$[0,1]$. The $\com^*$-action lifts
canonically to $\overline{Q}_{g,2}(\proj(V),d)$
and $S_U$. For the first $S_U$ in the integrand
\eqref{gt123}, we use the canonical lifting of $\com^*$.
For the second $S_U$, we tensor by a trivial line bundle
with fiber weights $-1$ over the two 
$\com^*$-fixed points of $\proj (V)$.
The classes $H$ are lifted to the distinct $\com^*$-fixed
points on $\proj(V)$.

The above choice of $\com^*$-action on the integrand
exactly parallels the choice of $\com^*$-action taken in \cite{FP}
for the analogous Gromov-Witten calculation.
The vanishing obtained in \cite{FP} also applies
for the stable quotient calculation here. 
The only loci with non-vanishing contribution to
the localization sum consist of two vertices
of genera 
$$g_1+g_2=g$$ connected by a single edge of
degree $d$.
The moduli spaces at these vertices are
$\overline{M}_{g_i,2|0}$ where
\begin{enumerate}
\item[(i)] the first two points are the respective node and marking,
\item[(ii)] there are no markings after the bar by vanishing.
\end{enumerate}

We find that the only non-vanishing contributions occur on
$\com^*$-fixed loci where the moduli of
stable quotients and the moduli of stable  maps 
are isomorphic. Moreover, on these 
loci, the bundle $R^1{\pi_*}(S_U)$
agrees with the analogous Gromov-Witten bundle.
Hence, the stable quotient integral $N_{g,d}$
is equal to the Gromov-Witten calculation of the
conifold \cite {FP}.
\end{proof}

The matching is somewhat of a surprise.
While the virtual classes of the stable quotient and
stable maps spaces to $\proj^1$ are related by Theorem 3,
the bundles in the respective integrands
for the conifold geometry are {\em not} compatible.
However, the differences happen away from the
non-vanishing loci.

If $g\geq 1$, no point insertions are required for
stability. The associated conifold integral is more subtle to
calculate, but the same result is obtained. We leave 
the details to the reader.{\footnote{The vanishing,
as before, matches the $\com^*$-fixed point
loci of the stable quotients and stable
maps spaces. However, the two which correspond
to a single vertex of genus $g$ are now not obviously
equal.
The match for these is obtained by redoing
the pointed integral \eqref{gt123}
 with both $H$ classes
in the integrand taken to lie over the {\em same} $\com^*$-fixed
point.}}

\begin{Proposition} For $g\geq 1$, 
$$N_{g,d} = 
\int_{[\overline{Q}_{g,0}(\proj^1,d)]^{vir}}
\mathsf{e}( R^1\pi_*(S_U) \oplus R^1\pi_*(S_U)).$$
\end{Proposition}

There are many other well-defined local toric Calabi-Yau geometries
to consider for stable quotients
both in dimension 3 and higher \cite{KP, ZP}.
The simplest is local $\proj^2$.

\begin{Question}
 What is the answer for the stable quotient theory
for 
$$\oh_{\proj^2}(-3) \rarr \proj^2\  \ ?$$
\end{Question}

\section{Other targets}\label{varr3}
\subsection{Virtual classes} \label{virc2}
Let $X\subset \proj^{n}$ be a projective
variety. There is a naturally associated substack
\begin{equation}\label{bbw}
\overline{Q}_{g,m}(X,d) \subset 
\overline{Q}_{g,m}(\proj^n,d)
\end{equation} 
defined by the following principle.
Let $I\subset \com[z_0,\ldots,z_n]$ be the homogeneous
ideal of $X$.
Given an element
\begin{equation}\label{hht}
(C,\ p_1,\ldots, p_m,\  0\rarr S \rarr
\com^{n+1}\otimes \oh_C \stackrel{q}{\rarr} Q \rarr 0)
\end{equation}
of $\overline{Q}_{g,m}(\proj^n,d)$, consider the
dual
$$\com^{n+1} \otimes \oh_C \stackrel{q^*}{\rarr} S^*$$
as a line bundle with $n+1$ sections $s_0,\ldots,s_n$. 
The stable quotient \eqref{hht} 
lies in
 $\overline{Q}_{g,m}(X,d)$
if for every homogeneous degree $k$ polynomial $f_k\in I$,
\begin{equation}\label{ggr}
f_k(s_0,\ldots,s_n) = 0 \in H^0(C,S^{k*}).
\end{equation}
Condition \eqref{ggr} is certainly well-defined
in families and determines a Deligne-Mumford 
substack. Local equations for
the substack \eqref{bbw} can easily be found.

\begin{Question}
If $X$ is nonsingular, does
$\overline{Q}_{g,m}(X,d)$ carry a canonical 2-term perfect
obstruction theory?
\end{Question}

The moduli space 
$\overline{Q}_{g,m}(X,d)$ depends upon the projective
embedding of $X$. If 
$\overline{Q}_{g,m}(X,d)$ does carry a virtual
class, the theory will almost certainly differ
somewhat from the Gromov-Witten counts.

If $X\subset \proj^n$ is nonsingular complete intersection,
more definite claims can be made.
For simplicity, assume $X$ is a hypersurface
defined by a degree $k$ equation $F$.
Given an element
\begin{equation*}
(C,\ p_1,\ldots, p_m,\  0\rarr S \rarr
\com^{n+1}\otimes \oh_C \stackrel{q}{\rarr} Q \rarr 0)
\end{equation*}
of $\overline{Q}_{g,m}(X,d)$,
the pull-back to $C$ of the tangent bundle to $X$
may be viewed as the complex
\begin{equation}
\label{hew}
S^*\otimes Q \stackrel{dF}{\rarr} S^{k*}
\end{equation}
defined by differentiation of the section $F$ on the zero
locus. 
We speculate an obstruction theory 
on $\overline{Q}_{g,m}(X,d)$
can be defined by the
hypercohomology of the sequence \eqref{hew}.
The 2-term condition follows from the fact
that the map $dF$ has cokernel with dimension
0 support. Many details have to be checked here.

\subsection{Elliptic invariants}\label{elt}
An interesting example to consider is
the moduli space $\overline{Q}_{1,0}(X_{n+1}\subset \proj^n,d)$
of stable quotients associated to the Calabi-Yau
hypersurfaces $X_{n+1}\subset \proj^n$.

By Proposition \ref{nn245},
$\overline{Q}_{1,0}(\proj^n,d)$
is a nonsingular space of expected dimension $(n+1)d$.
As before, let $S_U$ be the universal subsheaf on the universal
curve 
$$\pi:U \rarr \overline{Q}_{1,0}(\proj^n,d).$$
 Since
$S_U$ is locally free of rank 1, $S_U$ is a line bundle.
By the vanishing used in the proof of Proposition \ref{nn245}, 
$$\pi_* S^{*(n+1)}_U  \rarr
\overline{Q}_{1,0}(\proj^n,d)$$
is locally free of rank $(n+1)d$.

We define the genus 1 stable quotient invariants of $X_{n+1}\subset \proj^n$
by the integral
\begin{equation}\label{g45}
N^{X_{n+1}}_{1,d} = \int_{\overline{Q}_{1,0}(\proj^n,d)}
\mathsf{e}\left(\pi_* S^{*(n+1)}_U \right).
\end{equation}
The definition of $N^{X_{n+1}}_{1,d}$ is compatible
with the discussion of the virtual classes of hypersurfaces
in Section \ref{virc2}.

The genus 1 Gromov-Witten theory of
hypersurfaces has recently been solved by Zinger \cite{Zing}. 
Substantial work is required to
convert the Gromov-Witten calculation to an Euler
class on a space of genus 1 maps to projective space.
The stable quotient invariants
are immediately given by such an Euler class. There is no obstruction 
to calculating \eqref{g45} by
localization. 

\begin{Question}
What is the relationship between the stable quotient
and stable map invariants in genus 1 for Calabi-Yau
hypersurfaces?
\end{Question}

\subsection {Variants}\label{vart}
There are several variants which can be immediately
considered. Let $X$ be a nonsingular projective variety
with an ample line bundle $L$.
The stable quotient construction can be carried out
over the moduli space of stable maps $\overline{M}_{g,m}(X,\beta)$
instead of the moduli space of curves $\overline{M}_{g,m}$.
An object then consists of  three pieces of data:
\begin{enumerate}
\item[(i)] a genus $g$, $m$-pointed, quasi-stable curve $(C,p_1,\ldots,p_m)$,
\item[(ii)] a map $f:C \rarr X$ representing class 
$\beta\in H_2(X,\mathbb{Z})$,
\item[(iii)] and a quasi-stable quotient sequence
$$0 \rarr S \rarr \com^n \otimes \oh_C \rarr Q \rarr 0.$$
\end{enumerate}
Stability is defined by the ampleness of
$$\omega_C(p_1\ldots+p_m) \otimes f^*(L^3) 
\otimes (\wedge^{r} S^*)^{\otimes \epsilon}$$
on $C$ for every strictly positive $\epsilon\in \mathbb{Q}$.
We leave the details to the reader. The moduli space
is independent of the choice of $L$.

The  
moduli space carries a 2-term obstruction theory and
a virtual class. The corresponding descendent theory is
equivalent to the Gromov-Witten theory of $X\times \mathbb{G}(r,n)$
by straightforward modification of 
the arguments used to prove Theorem 4.

There is no reason to restrict to the trivial bundle in
(iii) above.
We may fix a rank $n$ vector bundle
$$B \rarr X$$
and replace the quasi-stable quotient sequence by
$$0 \rarr S \rarr f^*(B) \rarr Q \rarr 0.$$
The corresponding theory is perhaps equivalent to the
Gromov-Witten theory of the Grassmannian bundle over $X$
associated to $B$.
As $B$ may not split, a torus action may not
be available. The strategy of the proof of Theorem 4 does 
not directly apply.

A stranger replacement of the trivial bundle can be made even when
$X$ is a point. We may choose the quotient sequence to be
$$0 \rarr S \rarr H^0(C,\omega_C) \otimes \oh_C \rarr Q \rarr 0.$$
The middle term is essentially the pull-back of the Hodge bundle from
the moduli space of curves. 

\begin{Question}
What do integrals over the 
moduli of  stable Hodge quotients correspond to 
in Gromov-Witten theory?
\end{Question}

\vspace{+8 pt}
\noindent
Department of Mathematics\\
University of Illinois at Chicago\\
alina@math.uic.edu

\vspace{+8 pt}
\noindent
Department of Mathematics\\
University of California, San Diego\\
doprea@math.ucsd.edu

\vspace{+8 pt}
\noindent
Department of Mathematics\\
Princeton University\\
rahulp@math.princeton.edu.


\begin{thebibliography}{}

\bibitem{ACGH} E. Arbarelo, M. Cornalba, P. Griffiths, J. Harris,
{\em Geometry of algebraic curves}, Springer-Verlag, Berlin, 1985.
 
\bibitem{BF}
K.~Behrend, B.~Fantechi,
\newblock {\em The intrinsic normal cone,} 
{Invent. Math.} {\bf 128} (1997), 45--88.


\bibitem{CFK} I.~Ciocan-Fontanine, M.~Kapranov, 
\newblock {\em Virtual fundamental classes via dg-manifolds},  Geom. Topol.  {\bf 13}  (2009),  1779--1804. 

\bibitem{tor} I.~Ciocan-Fontanine, B.~Kim, 
\newblock {\em Moduli stacks of stable toric quasimaps}, Adv. Math. {\bf 225} (2010), 3022--3051.

\bibitem{C} Y. Cooper, \newblock{\em The geometry of stable
quotients in genus 1}, in preparation.

\bibitem {Faber}

C.~Faber, 
\newblock{\em A conjectural description of the tautological ring of the moduli space of curves}, Moduli of curves and abelian varieties,  109--129, Aspects Math., Vieweg, Braunschweig, 1999.



\bibitem{FP} 

C. Faber, R. Pandharipande,
\newblock{\em Hodge integrals and Gromov-Witten theory}, Invent.\ Math.\
{\bf 139} (2000), 173-199.

\bibitem{Ful} W. Fulton, {\em Intersection theory}, Springer-Verlag,
Berlin, 1984.

\bibitem{GetPan} 

E.~Getzler, R.~Pandharipande, 
\newblock{\em The Betti numbers of $\overline{\mathcal M}_{0,n}(r,d)$},  J. Algebraic Geom.  {\bf 15}  (2006), 709--732.


\bibitem{Giv} 

A. Givental, 
\newblock{\em Equivariant Gromov-Witten invariants}, Internat. Math. Res. Notices {\bf 13} (1996),  613--663.


\bibitem{GP}

T.~Graber, R.~Pandharipande,
\newblock {\em Localization of virtual classes}, Invent. Math. {\bf 135}
(1999),  487--518.


\bibitem{Has} 

B.~Hassett, 
\newblock {\em Moduli spaces of weighted pointed stable curves},  Adv. Math.  {\bf 173}  (2003), 316--352. 


\bibitem {HL}

Y.~Hu, J.~Li, 
\newblock{\em Genus-1 stable maps, local equations, and Vakil-Zinger's desingularization}, Math. Ann. {\bf 348} (2010), no. 4, 929Ð-963. 


\bibitem{J}  

D.~Joyce, 
\newblock {\em Configurations in abelian categories. IV. Invariants and changing stability conditions,}  Adv. Math.  {\bf 217}  (2008),  125--204.

\bibitem{KKO}

B.~Kim, A.~Kresch, Y.-G.~Oh, 
\newblock {\em A compactification of the space of maps from curves}, preprint (2007).


\bibitem {BP}

B.~Kim, R.~Pandharipande, 
\newblock {\em The connectedness of the moduli space of maps to homogeneous spaces},   Symplectic geometry and mirror symmetry, 187--201, World Sci. Publ., River Edge, NJ, 2001.


\bibitem {KP}

A.~Klemm, R.~Pandharipande, 
\newblock {\em Enumerative geometry of Calabi-Yau 4-folds},  Comm. Math. Phys.  {\bf 281}  (2008),  621--653.


\bibitem{KS} 

M.~Kontsevich, Y.~Soibelman, 
\newblock{\em  Stability structures, motivic Donaldson-Thomas invariants and cluster transformations}, arXiv:0811.2435.


\bibitem{Park}

J. ~Lee, T.~Parker, 
\newblock {\em A structure theorem for the Gromov-Witten
invariants of K\"ahler surfaces}, J. Differential Geom. {\bf 77}  (2007), 483--513.



\bibitem{LT}

J.~Li, G.~Tian,
\newblock {\em Virtual moduli cycles and {G}romov-{W}itten invariants of algebraic
  varieties,} { J. Amer. Math. Soc.}  {\bf 11} (1998), 119--174, 1998.

\bibitem{LLY}

B.~Lian, K.~Liu, S.~T.~Yau, 
\newblock{\em Mirror principle. I}, Surveys in differential geometry: differential geometry inspired by string theory,  405--454, Surv. Differ. Geom., 5, Int. Press, Boston, MA, 1999.


\bibitem {L}

E.~Looijenga, \newblock {\em On the tautological ring of ${\mathcal M}_g$}, Invent. Math. {\bf 121} (1995), 411--419. 


\bibitem{LM} 

A.~Losev, Y~Manin, 
\newblock{\em New moduli spaces of pointed curves and pencils of flat connections}, Michigan Math. J.  {\bf 48} (2000), 443--472. 

\bibitem{Mano} C. Manolache, \newblock{\em Virtual push-forwards},
arXiv:1010.2704.


\bibitem{MO} 

A.~Marian, D.~Oprea, 
\newblock{\em Virtual intersections on the Quot scheme and Vafa-Intriligator formulas},  Duke Math. J.  {\bf 136}  (2007), 81--113. 


\bibitem{MNOP1}

D.~Maulik, N.~Nekrasov, A.~Okounkov, and R.~Pandharipande,
\newblock {\em 
Gromov-{W}itten theory and {D}onaldson-{T}homas theory. {I},} 
{Compos. Math.} {\bf 142} (2006), 1263--1285.


\bibitem{MNOP2}

D.~Maulik, N.~Nekrasov, A.~Okounkov, and R.~Pandharipande.
\newblock {\em Gromov-{W}itten theory and {D}onaldson-{T}homas theory. {II}},
Compos. Math. {\bf 142} (2006), 1286--1304.

\bibitem {OMS}

C.~Okonek, M.~Schneider, H.~Spindler, \newblock{ \em Vector bundles on complex
projective spaces}, Progress in Mathematics, BirkhŠuser, Boston, 1980.


\bibitem {PP3}

R.~Pandharipande, 
\newblock{\em A compactification over $\overline{M}_g$ of the
universal moduli space of slope-semistable vector bundles},
JAMS {\bf 9} (1996), 425--471. 




\bibitem {P}

R.~Pandharipande, 
\newblock{\em Hodge integrals and degenerate contributions}, Comm. Math. Phys.  {\bf 208}  (1999),  489--506.

\bibitem {PP}  R.~Pandharipande, 
{\em The kappa ring of the moduli space of
curves of compact type I \& II},
arXiv:0906.2657    \&  arXiv:0906:2658.

\bibitem {PPP} R.~Pandharipande, A.~Pixton, {\em Relations in the tautological ring}, arXiv:1101.2236. 

\bibitem{PT}

R.~Pandharipande, R.~P. Thomas, 
\newblock {\em Curve counting via stable pairs in the derived category}, 
Invent. Math.
{\bf{178}} (2009), 407--447.



\bibitem {ZP}

R.~Pandharipande, A.~Zinger, 
\newblock {\em Enumerative Geometry of Calabi-Yau 5-Folds},
Adv. Studies in Pure Math. (to appear). 


\bibitem{Popa}

M.~Popa, M.~Roth, 
\newblock {\em Stable maps and Quot schemes}, Invent. Math. {\bf 152} (2003),  3, 625--663.


\bibitem{T}

R.~P. Thomas,
\newblock {\em A holomorphic {C}asson invariant for {C}alabi-{Y}au 3-folds, and
  bundles on {$K3$} fibrations,} {J. Differential Geom.}  {\bf 54} (2000),
  367--438.


\bibitem {Jes}

J.~Thomsen, 
\newblock{\em Irreducibility of $\overline{M}_{0,n}(G/P,\beta)$},  Internat. J. Math. {\bf 9}  (1998),  367--376.


\bibitem{Taub} 

C.~Taubes, 
\newblock{\em  ${\rm GR}={\rm SW}$: counting curves and connections},
J. Diff. Geom. {\bf 52} (1999),  453--609.

\bibitem{Toda}
Y. Toda, {\em Moduli spaces of stable quotients and the wall-crossing phenomena}, arXiv:1005.3743.

\bibitem{VZ} 

R.~Vakil, A.~Zinger,
\newblock {\em A desingularization of the main component of the moduli space of genus-one stable maps into $\mathbb P^n$},  Geom. Topol.  {\bf 12}  (2008),  1--95.


\bibitem {Zing}

A.~Zinger, 
\newblock {\em The reduced genus $1$ Gromov-Witten invariants of Calabi-Yau hypersurfaces},  {\bf 22}  (2009), 691--737.

\end{thebibliography}
\end{document}